\documentclass[oneside]{amsart}
\usepackage{amssymb,amsfonts}
\usepackage{amsthm}
\usepackage[square,numbers]{natbib}
\usepackage{amsmath}
\usepackage{graphicx}
\usepackage{marginnote}
\usepackage{color}
\usepackage[active]{srcltx}
\usepackage{multirow}
\usepackage[notref,notcite]{showkeys}
\usepackage{hyperref}
\usepackage[yyyymmdd]{datetime}
\usepackage{authblk}
\numberwithin{equation}{section}

\newtheorem{theorem}{Theorem}[section]
\newtheorem{lemma}[theorem]{Lemma}

\newtheorem{proposition}[theorem]{Proposition}
\newtheorem{definition}[theorem]{Definition}
\newtheorem*{theorem*}{Theorem}

\theoremstyle{remark}
\newtheorem{remark}[theorem]{\bf{Remark}}

\newcommand{\pbar}{\overline p}
\newcommand{\pt}{\widetilde p}

\newcommand{\zbar}{\overline z}
\newcommand{\zwave}{\widetilde z}
\newcommand{\e}{\varepsilon}
\newcommand{\ah}{\alpha}
\newcommand{\D}{\Delta}
\newcommand{\n}{\nabla}
\newcommand{\g}{\gamma}
\newcommand{\G}{\Gamma}

 \title{Well Productivity Index for Compressible Fluids and Gases} 
\newcommand{\authorfootnotes}{\renewcommand\thefootnote{\fnsymbol{footnote}}}

\begin{document}
\maketitle
\begin{center}
  \normalsize
  \authorfootnotes
  EUGENIO AULISA\footnote{{\it E-mail}: eugenio.aulisa@ttu.edu}\textsuperscript{1}, LIDIA BLOSHANSKAYA\footnote{{\it E-mail}: bloshanl@newpaltz.edu}
  \textsuperscript{2},
  AKIF IBRAGIMOV\footnote{{\it E-mail}: akif.ibraguimov@ttu.edu}\textsuperscript{1}\par \vspace{0.5cm}

  \textsuperscript{1}{\small Texas Tech University, Department of Mathematics and Statistics, Broadway and Boston, Lubbock, TX 79409-1042}\\
  \textsuperscript{2}{\small SUNY New Paltz, Department of Mathematics, 1 Hawk Dr, New Paltz, NY 12561}\par \bigskip

  \today\bigskip
\end{center}

\begin{abstract}
In this paper we discuss the notion of the diffusive capacity for the generalized Forchheimer flow of fluid through porous media. The diffusive capacity is an integral characteristic of the flow motivated by the engineering notion of the productivity index (PI), \cite{Dake, Raghavan, PI-gas}. The PI characterizes the well capacity with respect to drainage area of the well and in general is time dependent. We study its time dynamics for two types of fluids:  slightly compressible and strongly compressible fluid (ideal gas).  
In case of the slightly compressible fluid  the PI  stabilizes in time to the specific value, determined by the so-called pseudo steady state solution, \cite{ABI11, ABI12, AIVW09}. Here we generalize our results from \cite{ABI12} on long term dynamics of the PI in case of arbitrary order of the nonlinearity of the flow.

In this paper we   study the mathematical model of the  PI for compressible gas flow for the first time. In contrast to slightly compressible fluid this functional mathematically speaking is not time-invariant. At the same  time it stays "almost" constant for a long period of time, but then it rapidly blows up as time approaches the certain critical value. This value depends on the initial data (initial reserves) of the reservoir. The ``greater'' are the initial reserves, the larger is this critical value. We present numerical and analytical results for the time asymptotic of the PI and its stability with respect to the initial data. Using comparison theorems  for porous media equation from  \cite{VazquezPorousBook}  we obtain estimates between the PI's for the original gas flow and  auxiliary flow with a distributed source. The latter  one generates the time independent PI, and can be calculated using formula similar to one in case of slightly compressible fluid. 
\end{abstract}

\newpage
\tableofcontents

\section{Historical Remarks and Review of the Results}

The classical equation describing the fluid flow in porous media is the Darcy's law, stating the linear relation between the pressure gradient $\n p$ and velocity $u$.
Darcy himself observed in \cite{Darcy} that the area of applicability of linear relation is very limited. When, for instance, the fluid has high velocity or in the presence of fractures in the media, the nonlinear models are necessary to capture the properties of the flow.

One of the widely-used nonlinear models is the Forchheimer equation 
in the form $g(|u|)u=-\nabla p$, where $g(s)$ is a polynomial (see \cite{Bear, Muskat}). 
Originally Forchheimer in his work \cite{Forch} proposed three particular equations 
to match the experimental data: two term, three term and power laws, 
with  $g(s)$ being up to second order degree polynomial. 
To embrace the recent findings on the nonlinearity of the fluid flow 
(see \cite{li-engler, kazemi-ozkan}) and simplify the mathematical handling, 
it is convenient to consider the generalization of classical Forchheimer 
equations to the case where $g(s)$ is the generalized polynomial with non-negative 
coefficients and, possibly, non-integer powers (see \cite{ABHI1}). We call this family 
of equations the {\it $g$-Forchheimer equations}. 
The $g$-Forchheimer equation combined with the equation of state and the conservation 
of mass results in a single degenerate parabolic equation for the pressure $p(x,t)$, only.

In this paper we discuss two types of fluids: slightly compressible fluid, 
characterized by the equation of state $\rho(p)\sim \exp(\gamma p)$ 
with very small compressibility constant ($\gamma \sim10^{-8}$), 
and strongly compressible fluid, ideal gas, characterized by 
the equation of state $\rho(p)\sim p$, \cite{Muskat}. Note that in our model 
we assume that the porosity of the porous media does not depend on the pressure. 
In case of slightly compressible fluid we studied different properties of 
$g$-Forchheimer equations \cite{ABHI1, ABI12, HI12anydegree}. 
In this paper we extend the results of our work \cite{ABI12} 
on asymptotic behavior of the pressure function to the case when the degree 
of the $g$-polynomial is arbitrary. In case of ideal gas we discuss 
both numerical and analytical results for the time dynamics of the solution 
of the corresponding parabolic equation.


Keeping in mind that the applications of our findings are in geophysics 
and in particular in reservoir engineering, 
we restrict out studies to fluid flow in a reservior with bounded domain $U$. 
The reservoir  is bounded by the exterior impermeable boundary $\Gamma_e$
and the interior well-boundary $\Gamma_i$. $\Gamma_i$ 
is subject to various boundary conditions depending on the flow regime. 
We introduce the capacity type functional to study the asymptotic  behavior  
of the fluid flow in the reservoir with respect to time $t$. 
This functional is motivated by the notion of the well Productivity Index (PI) 
and is often used by reservoir engineers to measure well capacity, see \cite{Dake,Raghavan}. 
Productivity index is the total amount of fluid per unit pressure drawdown 
(the difference between reservoir and well pressures) that can be extracted 
by the well from a reservoir  (see \cite{Raghavan}). 
It is defined as $J_g(\Gamma_i)(t)=Q(t)/PDD(t)$, where $Q(t)$ is the total 
flux through the boundary $\Gamma_i$, and the pressure drawdown $PDD(t)$ 
is equal to the difference between averages of the pressure 
in the domain and on the boundary $\Gamma_i$.  

In general the diffusive capacity is time dependent (see \cite{Dake,Raghavan}) 
for both slightly compressible and compressible fluid flows. 
However its time dynamics differs greatly in these two cases. 

For slightly compressible fluids and specific regimes of production  
the PI stabilizes in time to a constant value. 
This value can be determined using the solution of a particular 
boundary value problem.  
Namely, the time-invariant diffusive 
capacity is associated to the pressure distribution 
$ p_s(x,t)=-\frac{Q_s}{V}t+W(x)$. 
Here the initial pressure distribution $W(x)$  
is the solution of the specific steady state BVP, $Q_s$  
is a constant flux through the well-boundary, and $V$ is the volume of 
the reservoir domain. Such pressure is called the 
{\it pseudo-steady state $($PSS$)$ pressure} and satisfies the split 
boundary condition on the well $-\frac{Q_s}{V}\,t+\varphi(x)$ 
for some known function $\varphi(x)$. 

For arbitrary initial data one can not expect the PSS pressure 
distribution and constant diffusive capacity. 
However, as it appears from the engineering  practice 
(see \cite{Dake, Raghavan}) the constant PSS PI serves 
as an attractor for the transient PI.  We proved this 
fact in \cite{ABI11, ABI12} under a number of conditions. 
In this paper we will generalize our previous results.  
We consider two types of boundary value problems. 

In the IBVP-I we impose total flux $Q(t)$ 
on the well-boundary and consider the trace of the 
pressure function to be split as $\gamma(t)+\psi(x,t)$, with $\int_{\Gamma_i}\psi\,ds=0$. 
In this case $\gamma(t)$ can be considered as the average of the 
trace function and $\psi$ as the deviation of the actual trace from its average. 
Note that we impose conditions only on the function $\psi$, while $\gamma$ is unknown. 
For $t\rightarrow \infty$ the boundary data $\{\psi(x,t), Q(t)\}$ are assumed 
to be localized in the neighborhood of the time independent $\{\varphi(x), Q_s\}$. 

In the IBVP-II no assumptions on the well-boundary flux are made. 
Instead, we specify the Dirichlet condition $\gamma(t)+\psi(x,t)$ 
on the well boundary. In this case both  $\gamma$ and $\psi$ are known functions.
For $t\rightarrow \infty$ the boundary data $\{\psi(x,t), \gamma(t)\}$ 
are assumed to be localized in the neighborhood of functions $\{\varphi(x), -\frac{Q_s}{V}t\}$.

The main result for both IBVP-I and II is that the time dependent 
diffusive capacity asymptotically stabilizes in the neighborhood  
of  the diffusive capacity for the PSS regime associated with the pair 
$\varphi(x)$ and $Q_s$ (see Theorems~\ref{th:pi-pipss-0} and \ref{th:pi-pipss-dir}). 
The corresponding value for the steady state diffusive capacity  
can be calculated just by solving an auxiliary Dirichlet time independent BVP.

In \cite{ABI12} we obtained this result under the {\it degree condition}, 
stating that the degree of $g$-polynomial ${\rm deg}(g)\leq\frac{4}{n-2}$, 
where $n$ is the dimension of space. Mathematically this condition arises 
from the theory of Sobolev spaces and insures the continuous embedding 
$W^{1,2-a}(U)\subset L^2(U)$. While  for $n=2$ this constraint holds 
for any degree ${\rm deg}(g)$, for $n=3$ it holds only for ${\rm deg}(g)\le4$. 
In this paper we study the asymptotic behavior of the diffusive capacity 
without any constraint on  ${\rm deg}(g)$. In this case the classical 
Poincar\'{e}-Sobolev inequality doesn't hold, and we use weighted 
inequality for the mixed term $|\n p|^{2-a}|p|^{\ah-2}$, $\ah\neq2$. 
We obtain estimates for the bounds of the $L^\ah$-norm for both the 
difference between transient and PSS solutions, and the difference 
between transient and PSS solution time derivatives.


In Sec. \ref{sec:Gas} we discuss the concept of the diffusive capacity for an ideal gas flow and some results on its time dynamics. In this case the productivity index is defined according to \cite{ aziz-gas-transient, PI-gas} as
$J_G(t)=Q(t)/\overline{PDD}(t)$, where $\overline{PDD}(t)$ is the difference between the average of $p^2$ in the domain $U$ and on the boundary $\Gamma_i$.  In contrast to the case of slightly compressible flow, 
we show numerically that until certain critical time $T_{crit}$  the transient PI remains almost 
constant  and then  as time approaches  $T_{crit}$ it blows up. This result obtained 
on actual field data corresponds with the engineering observations. 
Time $T_{crit}$ depends on the initial reserves/initial data: the ``greater'' 
the initial reserves are the greater  $T_{crit}$ is.
Similar to our approach in case of slightly compressible fluid,
we use the auxiliary pressure $p_0(x,t)$, resulting in time independent PI to investigate 
the behavior of time dependent PI.
The $p_0(x,t)$ is the solution of the equation with positive function on the RHS. 
This function can be considered as fluid injection inside the reservoir. 
When gas reserves are considerably larger than pressure drawdown on the well, t
his source term is negligible, and the PI's for $p_0(x,t)$  is almost identical 
with the general time dependent PI. 
For linear Darcy case we obtain some analytical comparison theorems 
between the pressures for time $t<T_{crit}$. Under some constraints 
on the boundary data and smoothness of the pressure function we obtain 
a stability result for the PI with respect to the initial data.


The paper is organized as follows. In Sec.~\ref{sec:problem-state} we discuss 
the various aspects of nonlinear Forchheimer equations and associated parabolic equation. 
We give the definition of the diffusive capacity on the solution of this parabolic equation. 
The corresponding IBVP-I and II for the total flux and Dirichlet boundary conditions 
are introduced in Sec.~\ref{sec:ibvp-dc-shift}. 
Sec.~\ref{sec:flux} is devoted to the IBVP-I: we state the conditions 
\eqref{assump1} - \eqref{assump4} imposed on the boundary data to obtain 
the main result on the asymptotic convergence of the transient diffusive capacity 
(see Theorem~\ref{th:pi-pipss-0}).
In Lemma~\ref{lem:npps0} we specify the estimates on the pressure, 
its gradient and time derivative that are necessary  to prove this result. 
Sec.~\ref{sec:dirichlet} is devoted to the similar results for IBVP-II: 
we state the conditions \eqref{assumpD1} - \eqref{assumpD3-2} imposed 
on the boundary data to obtain the asymptotic convergence of 
the transient diffusive capacity (see Theorem~\ref{th:pi-pipss-dir}). 
Finally, in Sec.~\ref{sec:Gas} we discuss the concept of PI for an ideal gas flow.


\section{Problem statement and Preliminary Properties}\label{sec:problem-state}
Consider a fluid in a porous medium occupying a bounded domain $U\subset\mathbb{R}^n$. 
Let $x\in\mathbb{R}^n$, $n\geq 2$, be a spatial variable  and $t\in\mathbb{R}$ be a time variable. 
Let $u(x,t)\in\mathbb{R}^n$ be the fluid velocity  and $p(x,t)\in\mathbb{R}$ be the pressure. 

We consider a generalized Forchheimer equation 
\begin{equation}\label{eq:g-forch}
 g(|u|)u=-\nabla p,
\end{equation}
where $g:\mathbb{R}^+\to\mathbb{R}^+$. In particular we consider function $g$ to be the
generalized polynomial with non-negative coefficients. Namely
\begin{equation}\label{gdef}
  g(s)=a_0s^{\alpha_0}+a_1s^{\alpha_1}+\dots+a_ks^{\alpha_k}=a_0+\sum_{j=1}^k a_js^{\alpha_j},
\end{equation}
with  $k\ge 0$,  the real exponents satisfy $\alpha_0=0<\alpha_1<\alpha_2<\ldots<\alpha_k$, and the coefficients $a_0,a_1,\ldots, a_k>0$.
The largest exponent $\alpha_k$ is the degree of $g$ and is denoted by $\deg(g)$.

One can notice that the equation \eqref{eq:g-forch} with $g(s)$ defined as in \eqref{gdef} includes the linear Darcy's equation  and all classical forms of Forchheimer equations \cite{Forch} (for details see our previous works \cite{ABHI1,ABI12}).

In case of slightly compressible fluid we consider the case when the Degree Condition in (2.14), \cite{ABI12}, is not satisfied, namely 
\begin{equation}\label{ndeg-cond}
a=\frac{\ah_k}{\ah_k+1}\leq\frac4{2+n}\quad\Longleftrightarrow\quad\alpha_k=\deg(g)>\frac4{n-2}.
\end{equation}
In this case there is no continuous embedding $W^{1,2-a}(U)\subset L^2(U)$ and corresponding Poincar\'{e} inequality doesn't hold. 
Clearly, if $n=3$ condition \eqref{ndeg-cond} will hold for the nonlinearities $\alpha_k>4$.

From \eqref{eq:g-forch} one can obtain the  non-linear Darcy equation explicitly solved for the velocity $u$:
\begin{equation}\label{non-lin_darcy}
 u=-K(|\nabla p|)\nabla p,
\end{equation}
where 
\begin{equation}\label{Kg}
 K(\xi)=\frac1{g(G^{-1}(\xi))},\quad \xi\geq0,\quad  G(s)=sg(s), \quad s\geq 0.
\end{equation}

Along with \eqref{eq:g-forch}, which is considered as a momentum equation, the dynamical system is subject to the continuity equation 
\begin{equation}\label{eq:continuity}
 \frac{\partial \rho}{\partial t}+\nabla\cdot(\rho u)=0,
\end{equation}
where $\rho(x,t)\in\mathbb{R}^+$ is density of the fluid.
The fluid is considered to be slightly compressible, subject to the equation of state
\begin{equation}\label{eq:state}
 \frac{\partial\rho}{\partial p}=\gamma\rho\quad\text{or}\quad \rho(p)=\rho_0 e^{\gamma(p-p_0)},
\end{equation}
 where $\gamma\sim10^{-8}$ is the compressiblity constant, see \cite{Muskat}.  

Combining \eqref{eq:g-forch}, \eqref{eq:continuity} and \eqref{eq:state} we obtain the degenerate parabolic equation for pressure (see \cite{ABHI1} for details):
\begin{equation}\label{eq:p-1}
 \gamma\frac{\partial p}{\partial t}=\n\cdot(K(|\n p|)\n p).
\end{equation}
By scaling the time variable we can assume further on that $\gamma=1$.

In case of the flow of ideal gas  
the equation \eqref{eq:g-forch} for $n=2$ will take the form (see, for example,  \cite{PayneStraughan-cont-conv})                                      
\begin{align*}
 \alpha u + \beta\rho u|u|=-\nabla p,\quad \beta=F\Phi k^{-1/2},
\end{align*}
As in case of slightly compressible fluid, the above equation can be solved for $u$:
\begin{equation*}
 u=-\frac{2}{\alpha+\sqrt{\alpha^2+4\beta\rho|\nabla p|}}\nabla p.
\end{equation*}
Combining the equation above with the equation of state $\rho(p)=p$ and continuity equation \eqref{eq:continuity} we obtain (for the details see Sec.~\ref{sec:Gas}) 
\begin{equation}\label{eq:gas-intro}
 \frac{\partial p}{\partial t}=
 \nabla\cdot\left(\frac{2 p }{\alpha+\sqrt{\alpha^2+4\beta p|\nabla p|}}\nabla p\right).
\end{equation}

\vspace{0.2cm}

We will study equations \eqref{eq:p-1} and \eqref{eq:gas-intro} in  the open domain $U\subset\mathbb{R}^n$  with the $C^2$ boundary $\partial U=\Gamma=\Gamma_e\cup \Gamma_i$. The $\Gamma_e$ is considered to be external impermeable boundary of the reservoir with $u\cdot N=0$, where $N$ is the outward normal vector to $\Gamma_e$. The $\Gamma_i$ is  the internal boundary of the well with the total flux or Dirichlet boundary conditions.

To study the long term dynamics of the $g$-Forchheimer flow in the domain $U$ we introduce special capacity type functional. We call it the  Diffusive Capacity/ Productivity Index $($PI$)$
(see \cite{Dake, Raghavan, wolf-durl-aziz})
and define as 
\begin{equation}\label{def:pi}
J_g(t)=\frac{Q(t)}{PDD(t)}.
\end{equation}
where $Q(t)$ is the total flux through the well-boundary $\Gamma_i$ and $PDD(t)$ is called the pressure drawdown in the domain $U$.\\
In case of slightly compressible fluid these quantities are deifned as
\begin{align}
  &Q(t)=\int_{\Gamma_i}u\cdot N\,ds,\\
  &PDD(t)=\pbar_U(t)-\pbar_{\Gamma_i}(t)=\frac1{|U|}\int_Up(x,t)dx- \frac1{|\Gamma_i|}\int_{\Gamma_i}p(x,t)ds.\label{def:av-pressure} 
\end{align}
are the pressure averages in the domain and on the well-boundary correspondingly

For gas filtration in porous media these quantities are deifned as, see \cite{ aziz-gas-transient, PI-gas}:
\begin{align*}
&Q(t)=\int_{\Gamma_i} \rho u \cdot N \, ds,\\
&PDD(t)=\frac1{|U|}\int_U p^2\,dx-\frac1{|\Gamma_i|}\int_{\Gamma_i} p^2\,ds.
\end{align*}

\vspace{0.3cm}

We end this section with the recollection of some  properties of $K(\xi)$ that will be used  in this study (see \cite{ABHI1, ABI12}).
Function $K(\xi)$, $\xi\in[0,\infty)$ is decreasing and satisfies
\begin{align}
0&\geq K'(\xi)\geq -a\,\frac{K(\xi)}{\xi},\label{k-prime}\\
 C_0(1+\xi)^{-a}&\leq K(\xi)\leq C_1(1+\xi)^{-a}\leq C_1,\label{k-ineq}\\
  C_2\xi^{2-a}-1&\leq K(\xi)\xi^2\leq C_1\xi^{2-a}.\label{k-xi2}
\end{align}
Moreover the function $K(|y|)y$, $y\in{\mathbb R}^n$, is monotone, i.e. for any two functions
$u_1, u_2\in W^{2,1}(U)$, one has
\begin{multline}\label{ineq:monoton}
\int_U (K(|\n u_1|)\n u_1-K(|\n u_2|)\n u_2)\cdot\n(u_1-u_2)\,dx\\
\geq C_\Phi^{-1}\left(\int_{U}|\nabla(u_1-u_2)|^{2-a}\,dx\right)^{\frac2{2-a}},
\end{multline}
where 
\begin{equation}\label{def:c-phi}
C_\Phi=C_\Phi(u_1, u_2)=C_1(1+\max(\|\nabla u_1\|_{L^{2-a}(U)},\|\nabla u_2\|_{L^{2-a}(U)}))^{a}.
\end{equation}

 Following \cite{ABHI1, ABI12} we define the functional 
\begin{equation}\label{h-def}
H(\xi)=\int_0^{\xi^2}K(\sqrt s)\,ds\quad\text{for}\quad \xi\geq0.
\end{equation}
Function $H(\xi)$ can be compared with $K(\xi)$ (see \cite{ABHI1}) as follows:
\begin{equation}\label{h-k-ineq}
 K(\xi)\xi^2\leq H(\xi)\leq 2K(\xi)\xi^2.
\end{equation}
Combining \eqref{h-k-ineq} with \eqref{k-xi2} there exist constants $C_1$ and $C_2$ such that
\begin{equation}\label{h-xi-2-a}
C_1\xi^{2-a}-1\leq H(\xi)\leq C_2\xi^{2-a}.
\end{equation}



\section{Two types of boundary conditions and Pseudo-Steady State solution}\label{sec:ibvp-dc-shift}
In Sections~\ref{sec:ibvp-dc-shift}, \ref{sec:flux} and \ref{sec:dirichlet} we consider the case of slightly compressible fluid.

Let $U\subset\mathbb{R}^n$ be an open set  with the $C^2$ boundary $\partial U=\Gamma=\Gamma_e\cup \Gamma_i$.  The $\Gamma_e$ is considered to be external impermeable boundary of the reservoir and the $\Gamma_i$ is  the internal boundary of the reservoir (well surface). With respect to boundary $\Gamma_i$   two different problems will be considered: IBVP-I with imposed total flux (generalized Neumann condition)  and IBVP-II with imposed Dirichlet boundary condition.


{\bf I.} {\it IBVP-I for the total flux condition.}
The function $p(x,t)$ is a solution of the IBVP-I if it satisfies
\begin{align}
& \frac{\partial p}{\partial t}=\n\cdot(K(|\n p|)\n p)\quad\text{in}\quad D=U\times(0,\infty),\label{eq:p}\\
& -\int_{\Gamma_i}K(|\n p|)\n p\cdot N\,ds=Q(t)
\quad\text{on}\quad \G_i\times(0,\infty),\label{bc:flux}\\
&\left.\frac{\partial p}{\partial N}\right|_{\G_e}=0\quad\text{on}\quad \G_e\times(0,\infty),\label{bc:0-neuman}\\
&p(x,0)=p_0(x)\quad\text{in}\quad U.\label{initialcond}
\end{align}
Since the  IBVP-I \eqref{eq:p}-\eqref{initialcond} lacks uniqueness, following \cite{ABI12} we restrict the boundary data. In particular,  let the trace of $p(x,t)$ on well-boundary $\Gamma_i$  be of the form
\begin{equation}\label{trace-p}
 p(x,t)|_{\Gamma_i}=\gamma(t)+\psi(x,t),
\end{equation}
with 
\begin{equation}
 \int_{\Gamma_i}\psi(x,t)\,ds=0.
\end{equation}
Note, that function $\gamma(t)$ is not specified and will be determined by the flux $Q(t)$. Such restriction of the boundary trace ensures the uniqueness of the solution of IBVP \eqref{eq:p}-\eqref{initialcond} (see Remark 3.2 in \cite{ABI12}). 

Transient boundary data $\psi(x,t)$ and $Q(t)$ will be compared with the time-independent boundary data
 $\varphi(x)$ and $Q_s$.
As in \cite{ABI12} we consider functions $\Psi(x,t)$ and $\Phi(x)$ to be defined in the domain $U$ with the traces on $\Gamma_i$ being $\psi(x,t)$ and $\varphi(x)$ correspondingly.
The results are obtained under the constraints on the 
parameters of difference of boundary data
\begin{equation}\label{deviation-flux}
\Delta_Q(t)=Q(t)-Q_s;\quad \Delta_\Psi(x,t)=\Psi(x,t)-\Phi(x).
\end{equation}

\vspace{0.2cm}

{\bf II.} {\it IBVP-II for Dirichlet boundary condition.}
The function $p(x,t)$ is a solution of the IBVP-II if it satisfies \eqref{eq:p}, \eqref{bc:0-neuman}, \eqref{initialcond} with the condition on the well boundary
\begin{equation}\label{bc:dir}
 p(x,t)|_{\Gamma_i}=\psi(x,t)+\gamma(t) \quad\text{on}\quad \G_i\times(0,\infty),
\end{equation}
for known function $\psi(x,t)$. 

Transient boundary data $\psi(x,t)$ and $\gamma(t)$ will be compared with the boundary data
 $\varphi(x)$ and $-At$, $A=Q_s/|U|$, where $Q_s$ is constant flux on $\Gamma_i$. We again consider 
functions $\Psi(x,t)$ and $\Phi(x)$ to be the $W_2^1(U)$ extensions of $\psi(x,t)$ and $\varphi(x)$ on the domain $U$.
The results are obtained under the constraints on the  differences
\begin{equation}\label{deviation-dir}
\Delta_\Psi(x,t)=\Psi(x,t)-\Phi(x); \quad \D_\gamma(t)=\gamma(t)+At.
\end{equation}

For the existence and regularity theory of degenerate parabolic equations, see e.g.~\cite{LadyParaBook68, Lions, HI12anydegree} and references therein.

\subsection{PSS solution}
For the  time-independent boundary data  $\varphi(x)$, $Q_s$ we define the
 {\it  pseudo steady state $($PSS$)$} solution of IBVP for equation \eqref{eq:p}
%
\begin{equation}\label{pss-profile}
 p_s(x,t)=-At + W(x), \quad A=\frac{Q_s}{|U|}
\end{equation}
for a given function $\varphi(x)$ and a constant total flux $Q_s$ on the boundary $\Gamma_i$.
Function $W(x)$ is called the \textit{basic profile} corresponding to the flux $Q_s$ and is defined as a solution of BVP
\begin{align}
& -A=\nabla\cdot(K(|\nabla W|)\nabla W),\label{eq:w}\\
& W(x)|_{\Gamma_i}=\varphi(x),\label{bc:w}\\
&\frac{\partial W}{\partial N}\Big|_{\Gamma_e}=0.\label{bc:w-neum}
\end{align}
The existence and uniqueness of a solution   for the given pair   $\varphi(x)$, $Q_s$ is proved in \cite{ABI11}.

It is not difficult to see that in case of the  PSS solution when $p(x,t)=p_s(x,t)$  the functional $J_g(t)$ in \eqref{def:pi} is time independent and 
\begin{equation}\label{def:pi-pss}
 J_g(t)=J_{g,PSS}=\cfrac{Q_s}{\frac1{|U|}\int_UW(x)\,dx-
\frac1{|\Gamma_i|}\int_{\Gamma_i}\varphi(x)\,ds}=const.
\end{equation}
We assume that $\varphi(x)$ and $Q_s$ are  such that the denominator of \eqref{def:pi-pss} is not equal zero. Obviously PSS solution is the only solution for given $\varphi(x)$, $Q_s$ which obey both types of boundary conditions. Therefore PSS will serve us a reference solution for comparison for both IBVP-I and II.

\subsection{Time convergence of $J_g(t)\to J_{g,PSS}$ - general case}
Let $p(x,t)$ be the transient solution of IBVP-I \eqref{eq:p}-\eqref{initialcond} with boundary data $\Psi(x,t), Q(t)$
or IBVP-II \eqref{eq:p}, \eqref{bc:0-neuman}, \eqref{initialcond}, \eqref{bc:dir} with boundary data $\Psi(x,t)$.
 Let $J_g(t)$ be the corresponding diffusive capacity defined as in \eqref{def:pi}.
Let
$p_s(x,t)$ be the PSS solution of this IBVP with boundary data $\Phi(x), Q_s$ and the corresponding constant diffusive capacity $J_{g,PSS}$ defined as in \eqref{def:pi-pss}.

 For slightly compressible fluid flow subjected to Froschheimer equation the aim of this study can be formulated as follows 
\begin{theorem}\label{thrm:main-gen}
 There is a wide class  of transient boundary data such that for both problem {\rm IBVP-I} and  {\rm IBVP-II} 
\begin{equation*}
 J_g(t)-J_{g,PSS}\to 0 \quad\text{as}\quad t\to\infty.
\end{equation*}
\end{theorem}

The proof uses the following  Lemma which reduces the question of the  convergence of $J(t)$ to the convergence of the gradient of the solution in the appropriate norm and of the total fluxes to a corresponding constant values. 

\begin{lemma}\label{lem:pi-diff}
As $t\to\infty$ the difference between transient and PSS PI's
\begin{equation}\label{cond:j-js}
 J_g(t)-J_{g,PSS}\to 0\quad\text{if}\quad \|\nabla(p-p_s)\|_{L^{2-a}(U)}\to 0\quad\text{and}\quad \Delta_Q(t)\to 0.
\end{equation}
\end{lemma}
\begin{proof}
From \eqref{def:pi} and \eqref{def:pi-pss} it follows
\begin{align}\label{j-jpss}
J_g(t)-J_{g,PSS}&= \frac{Q(t)}{\overline p_U-\overline p_{\Gamma_i}}- \frac{Q_s}{\overline p_{sU}-\overline p_{s{\Gamma_i}}}=
\frac{Q(t)(\overline p_{sU}-\overline p_{s{\Gamma_i}})-Q_s(\overline p_U-\overline p_{\Gamma_i})}{(\overline p_U-\overline p_{\Gamma_i})(\overline p_{sU}-\overline p_{s{\Gamma_i}})}\notag\\
&=\frac{Q(t)[(\overline p_{sU}-\overline p_{s{\Gamma_i}})-(\overline p_U-\overline p_{\Gamma_i})]+\Delta_Q(t)(\overline p_U-\overline p_{\Gamma_i})}{(\overline p_U-\overline p_{\Gamma_i})(\overline p_{sU}-\overline p_{s{\Gamma_i}})}\notag\\
&=\frac{Q(t)\cdot\Delta_p(t)}{(\overline p_U-\overline p_{\Gamma_i})(\overline p_{sU}-\overline p_{s{\Gamma_i}})}+
\Delta_Q(t)\,\frac{J_{g,PSS}}{Q_s}.
\end{align}
Here $\Delta_p(t)=(\pbar_{sU}-\pbar_{s\Gamma_i})-(\pbar_U-\pbar_{\Gamma_i})$
is the difference of pressure drawdowns for transient and PSS regimes.

The pressure drawdown for the PSS solution is time independent, and is not equal zero.
Hence $J_g(t)-J_{g,PSS}\to 0$ if and only if both $\Delta_p(t)\to 0$ and $\Delta_Q(t)\to 0$ as $t\to\infty$.

The difference in pressure drawdowns $\D_p(t)$ can be written as
\begin{align*}
 \Delta_p(t)=&
\ \frac{1}{|\Gamma_i|}\int_{\Gamma_i}(p-p_s)\,ds-\frac{1}{|U|}\int_U (p-p_s)\,dx\\
=&\ \frac{1}{|\Gamma_i|}\int_{\Gamma_i}\left(p-p_s-\frac1{|U|}\int_U(p-p_s)\,dx\right)\,ds\\
&-\frac{1}{|U|}\int_U \left(p-p_s-\frac1{|U|}\int_U(p-p_s)\,dx\right)\,dx
=\frac{1}{|\Gamma_i|}\int_{\Gamma_i}\zwave\,ds,
\end{align*}
where
\begin{equation*}
\zwave(x,t)=p-p_s-\frac1{|U|}\int_U(p-p_s)\,dx.
\end{equation*}

Applying Trace theorem  and Poincar\'{e} inequality, as $\int_U\zwave\,dx=0$, one has
\begin{align}
| \Delta_p(t)&|^{2-a}\leq 
C\int_{\Gamma_i}|\zwave|^{2-a}\,ds\leq
C\int_{U}|\nabla\zwave|^{2-a}\,dx=C\int_{U}|\nabla(p-p_s)|^{2-a}\,dx.\notag
\end{align}
The condition \eqref{cond:j-js} follows.
\end{proof}

 The results for the solution of IBVP-I will be obtained  under the conditions on the deviation of the boundary data \eqref{deviation-flux} in Sec.~\ref{sec:flux}, and for the solution of IBVP-II in terms of deviations \eqref{deviation-dir} in Sec.~\ref{sec:dirichlet}.



\section{IBVP-I: Asymptotic convergence of Diffusive Capacity}\label{sec:flux}

In this  section we prove the Theorem~\ref{thrm:main-gen} on asymptotic convergence of transient diffusive capacity $J_g(t)$ defined on the solution of IBVP-I \eqref{eq:p}-\eqref{initialcond} 
with boundary data $\Psi(x,t), Q(t)$.
Let
$p_s(x,t)$ be the PSS solution of this IBVP with boundary data $\Phi(x), Q_s$ and the corresponding constant diffusive capacity $J_{g,PSS}$.

\subsection{Assumptions on boundary data}

In order to formulate the assumptions in the Theorem~\ref{thrm:main-gen} we first introduce the following notations. 


Suppose $\ah\geq \frac{na}{2-a}>2$ and let  $2-a<b=\frac{\ah(2-a)}{2}<\ah$. We define
\begin{align}
 &F_1(t)=\frac1{|U|}\int_0^t\Delta_Q(\tau)\,d\tau+\frac1{|U|}\int_U\Delta_\Psi\,dx-\frac1{|U|}\int_U(p(x,0)-p_s(x,0))\,dx.\label{def:f1}\\
&F_2(t)=1+\|\nabla W\|_{L^{2-a}(U)}+|F_1(t)|+\int_U|\nabla(\Delta_\Psi)|\,dx+|\Delta_Q(t)|^{\frac{1}{1-a}}.\label{def:f2}
\end{align}
\begin{align}
&A_1(\ah,t)=\,\|\n W\|_{L^b}^{\ah-a}+\|\n(\D_\Psi)\|_{L^b}^{\ah-a}
+\bigg|\Delta_Q(t)+\int_U\Psi_t\,dx\bigg|^\ah\label{def:a1}\\
&\hspace{1.3cm}+\left(\int_U|\Psi_t|^\ah\,dx\right)^
{\frac{\ah-a}{\ah(1-a)}}+\D_Q^{\frac{\ah-a}{1-a}}(t).\notag\\
 &A_1(\ah)=\limsup\limits_{t\to\infty}A_1(\ah,t)^{\frac{\ah}{\ah-a}}.\label{def:limsupa1}
\end{align}
\begin{align}\label{def:a2}
A_{2}(t)= &\int_U|\nabla W|^{2-a}\,dx+\int_U|\nabla\Psi_t|^2\,dx+
\int_U|\nabla(\Delta_\Psi)|^{2-a}\,dx\\
&+\left(\int_U|\Psi_t|\,dx\right)^{2-a}+\int_U|\Psi_t|^2\,dx+\int_U|\Psi_t|^b\,dx\notag\\
&+|\Delta_Q(t)|+|Q'(t)|+|\Delta_Q(t)|^b+|Q'(t)|^b+C|F_1(t)|^{2-a}.\notag
\end{align}
\begin{equation}\label{def:a3}
 A_3(\e,t)=\int_U|\nabla\Psi_t|^2dx+\tfrac1{4\e}\int_U|\Psi_{tt}|^2\,dx+|Q'(t)|^2+|F'_1(t)|^2.
\end{equation}

The certain way the conditions A1-A5 will be formulated is dictated by the fact, that while the main result will be obtained under all the assumptions below, some intermidiate inequalities require less restrictive constraints.
As a note for each assumption we state the resulting bounds on parameters $F_1$, $F_2$, $A_1$, $A_2$ and $A_3$.\\

{\bf Assumptions 1 (A1):} {\small ($A_1(\ah)\leq C$)}
\begin{align} \label{assump1}
\limsup\limits_{t\to\infty}\Big(|\Delta_Q(t)|+ \|\Psi_t\|_{L^{b}}+
\|\n\Psi\|_{L^{b}}+\|\n\Phi\|_{L^{b}}\Big)\leq C.
\end{align}

{\bf Assumptions 2 (A2):} {\small ($A_1(\ah)+\limsup\limits_{t\to\infty}(A_2(t)+F_2(t))\leq C$)}
\begin{equation}\label{assump2}
\begin{cases}
 & \text{Assumptions $A1$; }\\
 &\limsup\limits_{t\to\infty}\Big(|Q'(t)|+|F_1(t)|+\|\nabla\Psi_t\|_{L^2}\Big)\leq C.
\end{cases} 
\end{equation}


{\bf Assumptions 3-$\beta$ (A3-$\beta$):} for any $\beta\geq2$
\begin{equation}\label{assump3-beta}
|\Delta'_{\gamma}(t)|+|\Delta_Q(t)|+\int_1^\infty \|\Psi_{tt}\|_{L^{\beta}(U)}+\|\nabla\Psi_t\|_{L^{\beta}(U)}^2\,dt
\leq C.
\end{equation}

{\bf Assumptions 4 (A4):} {\small $\limsup\limits_{t\to\infty}A_3(t)=0$ }
\begin{equation}\label{assump4}
\begin{cases}
& \text{Assumptions $A2$; }\\ 
&\limsup\limits_{t\to\infty}
\Big(\int_U(|\nabla\Psi_t|^2+|\Psi_{tt}|^2)\,dx+|Q'(t)|^2+\D^2_Q(t)\Big)=0.
\end{cases}
\end{equation}

{\bf Assumptions 5 (A5):}
\begin{equation}\label{assump5}
\begin{cases}
& \text{Assumptions $A3$ and $A4$; }\\ 
&\limsup\limits_{t\to\infty}\|\nabla(\Delta_\Psi)\|_{L^2}=0.
\end{cases}
\end{equation}


%
%
%
%
%

Along with the solution of IBVP-I $p(x,t)$ and the PSS solution $p_s(x,t)$ of IBVP-I with boundary data $\Phi(x), Q_s$, we will use the following shifts of the solutions:
\begin{align}
 &q(x,t)=p(x,t)-\frac1{|U|}\int_U p(x,t)\,dx-\left(\Psi(x,t)-\frac1{|U|}\int_U\Psi(x,t)\,dx\right);\label{def:q}\\
 &q_s(x)=p_s(x,t)-\frac1{|U|}\int_U p_s(x,t)\,dx-\left(\Phi(x)-\frac1{|U|}\int_U\Phi(x)\,dx\right).\label{def:qs}
\end{align}
It follows that $q$ satisfies the BVP
\begin{align}
 &q_t(x,t)=p_t(x,t)+\frac1{|U|}Q(t)-\left(\Psi_t(x,t)-\frac1{|U|}\int_U\Psi_t(x,t)\,dx\right)\label{eq:q},\\
 &q|_{\Gamma_i}=\gamma(t)-\frac1{|U|}\int_Up(x,0)dx+\frac1{|U|}\int_0^tQ(\tau)d\tau+\frac1{|U|}\int_U\Psi(x,t)\,dx.\label{bc:q}
\end{align}
Similarly $q_s$ satisfies the BVP
\begin{align}
 &\frac{\partial q_s}{\partial t}(x,t)=\frac{\partial p_s}{\partial t}(x,t)+A=0,\label{eq:qs}\\
 & q_s|_{\Gamma_i}=-\frac1{|U|}\int_U (p_s(x,0)-\Phi(x))\,dx.\label{bc:qs}
\end{align}
Note that
\begin{equation}\label{intq=0}
 \int_Uq(x,t)\,dx=0 \quad\quad\text{and}\quad\quad  \int_Uq_s(x,t)\,dx=0.
 \end{equation}

\vspace{0.2cm}

According to Lemma~\ref{lem:pi-diff} we need the convergence $\Delta_Q(t)\to 0$ and  $\|\nabla(p-p_s)\|_{L^{2-a}}\to 0$ as $t\to\infty$. The affinity of $Q(t)$ and $Q_s$ will be imposed as a condition on the boundary data (see Assumptions A4). The estimates necessary to prove the convergence of $\|\nabla(p-p_s)\|_{L^{2-a}}$ are outlined in the following Lemma. The references to the corresponding results  further in the paper are given.

\begin{lemma}\label{lem:npps0}
Let $q$ and $q_s$ be defined as in \eqref{def:q} and \eqref{def:qs}. Then
\begin{equation*}
\|\nabla(p-p_s)\|_{L^{2-a}} \to 0 \quad\text{as}\quad t\to\infty
\end{equation*}
 if
\begin{align}\label{cond:someto0}
 \lim_{t\to\infty}\left(\|\nabla(\Delta_\Psi)\|_{L^2}+|\D_Q(t)||F_1(t)+1|\right)=0
\end{align}
and 
\begin{align}
&1.\ \int_U|q-q_s|^2\,dx\leq C\int_U|q-q_s|^{\ah}\,dx\leq C, \quad \text{see Theorem~\ref{th:zbar-to0}};\label{ineq:zah-bound}\\
&2. \ \int_U|\n p|^{2-a}\,dx\leq C, \quad \text{see Theorem~\ref{thrm:bdd-h-gradp}}; \label{ineq:np-bound}\\
&3.\ \int_U|p_t+A|^2\,dx=\int_U|p_t-p_{s,t}|^2\,dx\to0,\quad \text{see Theorem \ref{thr:pt+A-to0}};\label{ineq:pta-0}\\
&4. \ \int_U|\n p_s|^{2-a}\,dx=\int_U|\n W|^{2-a}\,dx\leq C,\quad \text{see Lemma 4.1, \cite{ABI12}}.\label{ineq:nps-bound}
\end{align}
\end{lemma}
\begin{proof}
According to the inequality
(5.20) in proof of  Theorem~5.6 in \cite{ABI12} we have:
\begin{align}\label{n-p-ps-1}
&\|\nabla(p-p_s)\|_{L^{2-a}}^2\leq C_\Phi\int_U|p_t+A||q-q_s|\,dx \\
&\ \ \ +C_\Phi\left(\int_U|\nabla(\Delta_\Psi)|^2\,dx+\Delta^2_Q(t)+|F_1(t)\D_Q(t)|+\|\nabla(\Delta_\Psi)\|_{L^2}\right).\notag
\end{align}
Here $C_\Phi=C_\Phi(p, W)$ is as in \eqref{def:c-phi} and is bounded by virtue of  \eqref{ineq:np-bound}
and \eqref{ineq:nps-bound}.
It is then clear 
that the RHS of \eqref{n-p-ps-1} converges to 0 under conditions \eqref{cond:someto0}-\eqref{ineq:nps-bound}. 
\end{proof}




\subsection{Bounds for the solutions}
 As before let $q(x,t)$ be the solution of BVP \eqref{eq:q}, \eqref{bc:q} and $q_s(x,t)$ be the solution of PSS BVP \eqref{eq:qs}, \eqref{bc:qs}. Let $z=q-q_s$. We will prove that under certain condition $\int_U|z|^\ah\,dx$ is bounded at time infinity. 

In order to obtain the suitable differential inequality we will use the following result (see Lemma 2.3, \cite{HI12anydegree})
\begin{lemma}
Let $U$ be an open bounded domain in ${\mathbb R}^n$ and $\ah\geq\frac{na}{2-a}$. If $u|_{\partial U}=0$, there exists constant $C$ such that
\begin{align}\label{ineq:w-poincare-simple-0}
\int_UK(|\nabla u|)|\nabla u|^{2}|u|^{\alpha-2}\,dx&\geq C\int_U|\nabla u|^{2-a}|u|^{\alpha-2}-C\,dx\notag\\
&\geq C\left(\int_U|u|^\alpha\,dx\right)^{\frac1{\gamma_0}}-C,
\end{align}
where $\gamma_0=\alpha/(\alpha-a)$. 
\end{lemma}

We first prove the following lemma to estimate the boundary data. 

\begin{lemma}\label{lemma:bdd-for-q-qs}
Let $F_1(t)$ be as in \eqref{def:f1}, $\Delta_Q=Q(t)-Q_s$ and $\Delta_\gamma=\gamma(t)+At$.
Then for $\alpha\geq \frac{na}{2-a}$
 \begin{align}
  |\Delta_\gamma(t)+&F_1(t)|^{\ah-1}|\Delta_Q(t)|\\
&\leq \e_1 \left(\int_{U}|z|^{\ah}\,dx\right)^{\frac{1}{\g_0}}
+ \e_2 \int_{U}|\n z |^{2-a}|z|^{\ah-2}\,dx
+C\D_Q^{\frac{\ah-a}{1-a}}\notag
\end{align}
with $\gamma_0=\alpha/(\alpha-a)$.
\end{lemma}
\begin{proof}
 We have $ z |_{\Gamma_i}=(q-q_s)_{\Gamma_i}=\Delta_\gamma(t)+F_1(t)$.
Then by Trace theorem
\begin{align}
 ((\Delta_\gamma&+F_1)\cdot |\Gamma_i|)^{\ah-a}=\left(\int_{\G_i}| z |\,ds\right)^{\ah-a}
\leq C \int_{\G_i}| z |^{\ah-a}\,ds\\
&=C \int_{\G_i}| z |^{\frac{\ah-a}{2-a}\cdot (2-a)}\,ds
\leq C \int_{U}| z |^{\ah-a}\,dx
+ C \int_{U}\left[\nabla\left(| z |^{\frac{\ah-a}{2-a}}\right)\right]^{2-a}\,dx\notag\\
&= C \int_{U}| z |^{\ah-a}\,dx
+ C \int_{U}\left[|\n z || z |^{\frac{\ah-2}{2-a}}\right]^{2-a}\,dx\notag\\
&= C \int_{U}| z |^{\ah-a}\,dx
+ C \int_{U}|\n z |^{2-a}| z |^{\ah-2}\,dx.\notag
\end{align}
Then
\begin{align}
 |\Delta_\gamma+F_1|^{\ah-1}
\leq C \left(\int_{U}| z |^{\ah-a}\,dx\right)^{\frac{\ah-1}{\ah-a}}
+ C \left(\int_{U}|\n z |^{2-a}| z |^{\ah-2}\,dx\right)^{\frac{\ah-1}{\ah-a}}\\
\leq C \left(\int_{U}| z |^{\ah}\,dx\right)^{\frac{\ah-1}{\ah}}
+ C \left(\int_{U}|\n z |^{2-a}| z |^{\ah-2}\,dx\right)^{\frac{\ah-1}{\ah-a}}.\notag
\end{align}
Applying Young's inequality with $\e$ we get
\begin{align*}
 |\Delta_\gamma+F_1|^{\ah-1}|\Delta_Q(t)|
\leq \e_1 &\left(\int_{U}| z |^{\ah}\,dx\right)^{\frac{1}{\g_0}}
+ \e_2 \int_{U}|\n z |^{2-a}| z |^{\ah-2}\,dx\\
&+(C_{\e_1}+C_{\e_2})\D_Q^{\frac{\ah-a}{1-a}},
\end{align*}
which proves the result.
\end{proof}

\vspace{0.3cm}





\begin{lemma}\label{lem:ddt-q-qs-1}
Suppose $\ah\geq\frac{na}{2-a}>2$ and let $\gamma_0=\frac{\ah}{\ah-a}$. Then for $ z =q-q_s$ and $A_1(\ah,t)$ is as in \eqref{def:a1}
\begin{equation}\label{est:ddt-zbar}
  \frac{d}{dt}\int_U| z |^\ah\,dx\leq -C \left(\int_U| z |^\ah\,dx\right)^{\frac1{\gamma_0}}+C(1+A_1(\ah,t)).
 \end{equation}
\end{lemma}
\begin{proof}
Subtracting equations \eqref{eq:q} and \eqref{eq:qs} from each other, multiplying on 
$| z |^{\ah-1}{\rm sign}( z )$ and integrating over $U$ one has
\begin{align*}
 \frac1\alpha&\frac{d}{dt}\int_U| z |^\alpha\,dx=-(\ah-1)\int_U(K(|\nabla p|)\nabla p-K(|\nabla W|)\nabla W)\nabla z | z |^{\ah-2}\,dx\\
&+\int_{\Gamma_i}(K(|\nabla p|)\nabla p-K(|\nabla W|)\nabla W)\cdot N| z |^{\ah-1}\,ds +\frac1{|U|}\Delta_Q(t)\int_U| z |^{\ah-1}{\rm sign}( z )\,dx\\
&\quad\quad\quad\quad\quad\quad\quad\quad\quad
-\int_U\big[\Psi_t-\frac1{|U|}\int_U\Psi_t\,dy\big]| z |^{\ah-1}{\rm sign}( z )\,dx.
\end{align*}
From boundary conditions \eqref{bc:q} and \eqref{bc:qs} it follows that $ z |_{\Gamma_i}=\Delta_\gamma(t)+F_1(t)$, and one has
\begin{align}\label{ddt-z-0}
\frac1\ah&\frac{d}{dt}\int_U| z |^\ah\,dx=
-(\ah-1)\int_U(K(|\nabla p|)\nabla p-K(|\nabla W|)\nabla W)\nabla(p-p_s)| z |^{\ah-2}\,dx\notag\\
&+(\ah-1)\int_U(K(|\nabla p|)\nabla p-K(|\nabla W|)\nabla W)\nabla(\Delta_\Psi)| z |^{\ah-2}\,dx
\notag\\
&+\frac1{|U|}\int_U\big[\Delta_Q(t)+\int_U\Psi_t\,dx-|U|\Psi_t\big]| z |^{\ah-1}{\rm sign}( z )\,dx
+[\Delta_\gamma(t)+F_1(t)]^{\ah-1}\Delta_Q(t).
\end{align}

We split the first integral in RHS of \eqref{ddt-z-0} in four separate integrals and estimate them one-by-one. Namely,
\begin{equation}\label{4-eq}
-\int_U(K(|\nabla p|)\nabla p-K(|\nabla W|)\nabla W)\n(p-p_s)| z |^{\ah-2}\,dx= I_p+I_w+I_{pw}+I_{wp},
\end{equation}
where (notice that $\nabla p_s=\nabla W$)
\begin{equation*}
\begin{aligned}
&I_p=-\int_UK(|\nabla p|)|\nabla p|^2| z |^{\ah-2}\,dx, 
\ &&I_w=-\int_UK(|\nabla W|)|\nabla W|^2| z |^{\ah-2}\,dx\leq0, \\
&I_{pw}=\int_UK(|\nabla p|)(\nabla p\cdot\nabla W)| z |^{\ah-2}\,dx, 
 &&I_{wp}=\int_UK(|\nabla W|)(\nabla W\cdot\nabla p)| z |^{\ah-2}\,dx.
\end{aligned}
\end{equation*}
Since $\nabla z  = \nabla(p-W)-\nabla(\Delta_\Psi)$ and $\frac{\ah-2}{\ah}<\frac{\ah-a}{\ah}=\frac1{\gamma_0}$, similar to (4.5) in \cite{HI12anydegree} we have
\begin{align}\label{Ip-final}
I_p&\leq -C\int_U(|\nabla p|^{2-a}-1)| z |^{\ah-2}\,dx
\leq -C\int_U |\n  z |^{2-a}| z |^{\ah-2}\,dx\notag\\
&\qquad\qquad+C\int_U (|\n(\D_\Psi)|^{2-a}+|\n p_s|^{2-a}+1)| z |^{\ah-2}\,dx\notag\\
&\leq -C\int_U |\n z |^{2-a}| z |^{\ah-2}\,dx+
\e_1\left(\int_U| z |^\ah\,dx\right)^{\frac1{\gamma_0}}\notag\\
&\qquad\qquad+C_{\e_1}\left[1+
\|\n W\|_{L^b(U)}^{\ah-a}+\|\n(\Delta_\Psi)\|_{L^b(U)}^{\ah-a}
\right].
\end{align}
%

The integral $I_{pw}$ can be estimated similar to (4.6)-(4.9) in \cite{HI12anydegree}:
\begin{align}\label{Ipw}
 &I_{pw}\leq\int_UK(|\nabla p|)|\nabla p||\nabla W|| z |^{\ah-2}\,dx
\leq 
\int_U|\nabla p|^{1-a}|\nabla W|| z |^{\ah-2}\,dx\\
& \leq\int_U|\nabla z |^{1-a}|\nabla W|| z |^{\ah-2}\,dx
+ \int_U(|\nabla W|^{2-a}+|\nabla (\Delta_\Psi)|^{1-a}|\nabla W|)| z |^{\ah-2}\,dx.\notag
\end{align}
The first integral in \eqref{Ipw} can be estimated similar to (4.8)
\begin{align}
 \int_U&|\nabla z |^{1-a}|\nabla W|| z |^{\ah-2}\,dx\notag\\
&\leq \e\int_U|\nabla z |^{2-a}| z |^{\ah-2}\,dx
+\e\left(\int_U| z |^\ah\,dx\right)^{\frac1{\gamma_0}}+C_\e\, \|\n W\|_{L^b(U)}^{\ah-a}.
\end{align}
The second integral in \eqref{Ipw} is estimated similar to (4.7)
\begin{equation}
 \int_U|\nabla W|^{2-a}| z |^{\ah-2}\,dx
\leq \e\left(\int_U| z |^\ah\,dx\right)^{\frac1{\gamma_0}}+
C_\e\, \|\n W\|_{L^b(U)}^{\ah-a}.
\end{equation}
Similar the third integral in \eqref{Ipw} can be estimated as
\begin{align}\label{int-1-a}
 \int_U|\nabla (\Delta_\Psi)|^{1-a}|\nabla W|&|z|^{\ah-2}\,dx
\leq \, \e\left(\int_U| z |^\ah\,dx\right)^{\frac1{\gamma_0}}\\
&+C_\e \left(\int_U|\nabla (\Delta_\Psi)|^{\frac{\ah(1-a)}{2}}
|\nabla W|^{\frac\ah2}\,dx\right)^{\frac{\ah-a}{b}}.\notag
\end{align}
Since $\frac{1}{2-a}+\frac{1-a}{2-a}=1$, from H\"{o}lder and Young's inequalities it follows 
\begin{equation}
 \left(\int_U|\nabla (\Delta_\Psi)|^{\frac{\ah(1-a)}{2}}
|\nabla W|^{\frac\ah2}\,dx\right)^{\frac{\ah-a}{b}}\leq 
C\left(\|\n W\|_{L^b(U)}^{\ah-a}+\|\n(\Delta_\Psi)\|_{L^b(U)}^{\ah-a}\right).
\end{equation}
Combining the estimates above in \eqref{Ipw} we have the estimate for the integral $I_{pw}$
\begin{align}\label{Ipw-final}
 I_{pw}\leq \ \e\int_U|\nabla z |^{2-a}| z |^{\ah-2}\,dx
&+\e\left(\int_U| z |^\ah\,dx\right)^{\frac1{\gamma_0}}\notag\\
&+C\left(\|\n W\|_{L^b(U)}^{\ah-a}+\|\n(\Delta_\Psi)\|_{L^b(U)}^{\ah-a}\right).
\end{align}

Finally, for the integral $I_{wp}$ we have
\begin{align}\label{Iwp}
I_{wp}&\leq\int_UK(|\nabla W|)|\nabla W||\nabla p|| z |^{\ah-2}\,dx
\leq \int_U|\nabla W|^{1-a}|\nabla p|| z |^{\ah-2}\,dx\\
&\leq C\int_U|\nabla W|^{1-a}|\nabla  z || z |^{\ah-2}\,dx+
C\int_U|\nabla W|^{1-a}|\n (W+\D_\Psi)|| z |^{\ah-2}\,dx.\notag
\end{align}
Applying H\"{o}lder's and Young's inequalities for the first integral in the RHS of \eqref{Iwp} we get 
\begin{align}
 \int_U&|\nabla W|^{1-a}|\nabla  z || z |^{\ah-2}\,dx\notag\\
&=\int_U|\nabla W|^{1-a}
\cdot |\nabla  z || z |^{\frac{(\ah-2)}{(2-a)}}
\cdot| z |^{\frac{(\ah-2)(1-a)}{2-a}}\,dx\notag\\
&\leq 
\left(\int_U|\n  z |^{2-a}| z |^{\ah-2}\,dx\right)^{\frac{1}{2-a}}
\cdot \left(\int_U|\nabla W|^{2-a}| z |^{\ah-2}\,dx\right)^{\frac{1-a}{2-a}}\notag\\
&\leq 
 \e\int_U|\n  z |^{2-a}| z |^{\ah-2}\,dx 
+ C_\e\int_U|\nabla W|^{2-a}| z |^{\ah-2}\,dx \notag\\
&\leq \e\int_U|\n  z |^{2-a}| z |^{\ah-2}\,dx +
\e\left(\int_U| z |^\ah\,dx\right)^{\frac1{\gamma_0}}+
C_\e\, \|\n W\|_{L^b(U)}^{\ah-a}.\notag
\end{align}
The second integral in the RHS of \eqref{Iwp} can be estimated similar to \eqref{int-1-a}
\begin{align}
 \int_U|\nabla W|^{1-a}&|\n (W+\D_\Psi)|| z |^{\ah-2}\,dx\notag\\
&\leq \e\left(\int_U| z |^\ah\,dx\right)^{\frac1{\gamma_0}} +  
C\left(\|\n W\|_{L^b(U)}^{\ah-a}+\|\n(\Delta_\Psi)\|_{L^b(U)}^{\ah-a}\right). \notag
\end{align}
Then in \eqref{Iwp} the integral $I_{wp}$ can be estimated as
\begin{align}\label{Iwp-final}
 I_{wp}\leq \e\int_U|\n  z |^{2-a}| z |^{\ah-2}\,dx +
\e\left(\int_U| z |^\ah\,dx\right)^{\frac1{\gamma_0}}+
C\left(\|\n W\|_{L^b(U)}^{\ah-a}+\|\n(\Delta_\Psi)\|_{L^b(U)}^{\ah-a}\right).
\end{align}
\noindent Substiting \eqref{Ip-final}, \eqref{Ipw-final} and \eqref{Iwp-final} in \eqref{4-eq} we get the estimate for the first integral in the RHS of \eqref{ddt-z-0} 
\begin{align}\label{1st-int-rhs}
-\int_U(K(|\nabla p|)\nabla p-&K(|\nabla W|)\nabla W)\n(p-p_s)| z |^{\ah-2}\,dx\notag\\
 \leq&-(C-\e)\int_U|\n  z |^{2-a}| z |^{\ah-2}\,dx +
\e\left(\int_U| z |^\ah\,dx\right)^{\frac1{\gamma_0}}\notag\\
&+C\,\left[1+\|\n W\|_{L^b(U)}^{\ah-a}+\|\n(\Delta_\Psi)\|_{L^b(U)}^{\ah-a}\right].
\end{align}

\noindent The second integral in the RHS of \eqref{ddt-z-0} can be estimated as 
\begin{align}
&\int_U(K(|\nabla p|)\nabla p-K(|\nabla W|)\nabla W)\nabla(\Delta_\Psi)| z |^{\ah-2}\,dx\\
&\leq C \int_U K(|\nabla p|)|\nabla p||\nabla(\Delta_\Psi)|| z |^{\ah-2}\,dx
+\int_UK(|\nabla W|)|\nabla W||\nabla(\Delta_\Psi)|| z |^{\ah-2}\,dx.\notag
\end{align}
Then similar to the estimate for $I_{pw}$ and \eqref{int-1-a} we get
\begin{align}\label{2nd-int-rhs}
\int_U(K(|\nabla p|)\nabla p&-K(|\nabla W|)\nabla W)\nabla(\Delta_\Psi)| z |^{\ah-2}\,dx\leq \e\int_U|\nabla z |^{2-a}| z |^{\ah-2}\,dx\notag\\
&+\e\left(\int_U| z |^\ah\,dx\right)^{\frac1{\gamma_0}}+C\left(\|\n W\|_{L^b(U)}^{\ah-a}+\|\n(\Delta_\Psi)\|_{L^b(U)}^{\ah-a}\right).
\end{align}
By H\"{o}lder's and Young's inequalities, we get
\begin{align}\label{some-bdd-est}
 &\frac1{|U|}\int_U\big[\Delta_Q(t)+\int_U\Psi_t\,dx-|U|\cdot\Psi_t\big]| z |^{\ah-1}{\rm sign}( z )\,dx\\
&\leq \e\left(\int_U| z |^\ah\,dx\right)^{\frac1{\gamma_0}}+
C\left(\int_U\bigg|\Delta_Q(t)+\int_U\Psi_t\,dx-|U|\cdot\Psi_t\bigg|^\ah\,dx\right)^
{\frac{\ah-a}{\ah(1-a)}}\notag\\
&\leq \e\left(\int_U| z |^\ah\,dx\right)^{\frac1{\gamma_0}}+
C\bigg|\Delta_Q(t)+\int_U\Psi_t\,dx\bigg|^\ah+C\left(\int_U|\Psi_t|^\ah\,dx\right)^
{\frac{\ah-a}{\ah(1-a)}}.\notag
\end{align}

Finally combining \eqref{1st-int-rhs}, \eqref{2nd-int-rhs}, \eqref{some-bdd-est} and Lemma~\ref{lemma:bdd-for-q-qs} in \eqref{ddt-z-0} we get after choosing small enough $\e$
\begin{align}\label{yet-another-formula}
 \frac{d}{dt}\int_U| z |^\ah\,dx\leq -C\int_U|\nabla z |^{2-a}| z |^{\ah-2}\,dx
+\e\left(\int_U| z |^\ah\,dx\right)^{\frac1{\gamma_0}}+C(1+A_1(\ah,t)).
\end{align}

Applying \eqref{ineq:w-poincare-simple-0} to the first term in the RHS of \eqref{yet-another-formula} we get 
\begin{align*}
 \frac{d}{dt}\int_U| z |^\ah\,dx\leq -C\left(\left(\int_U| z |^\ah\,dx\right)^{\frac1{\gamma_0}}-1\right)
+\e\left(\int_U| z |^\ah\,dx\right)^{\frac1{\gamma_0}}+C(1+A_1(\ah,t)).
\end{align*}
Selecting sufficiently small $\e$ we obtain \eqref{est:ddt-zbar}.
\end{proof}


\vspace{.5cm}

In order to prove the estimate on $\int_U| z |^\ah\,dx$ we recall the Lemma A.1 from \cite{HI12anydegree}.
\begin{lemma}\label{lem:A1-gronw}
 Let $\phi:[0,\infty)\to[0,\infty)$ be a continuous, strictly increasing function. Suppose $y(t)\geq0$ is a continuous function on $[0,\infty)$ such that
 \begin{equation*}
  y'\leq-C\phi^{-1}(y(t))+f(t), \quad t>0,
 \end{equation*}
where $f(t)\geq0$ for $t\geq0$ is a  continuous function. Then 
$$\limsup\limits_{t\to\infty}y(t)\leq \phi\left(\limsup\limits_{t\to\infty}\frac{f(t)}{h(t)}\right)$$.
\end{lemma}


\vspace{0.5cm}

The following result will follow from Lemma~\ref{lem:ddt-q-qs-1}.
\begin{theorem}\label{th:zbar-to0}
For $\ah\geq na/(2-a)$ and $A_1(\ah)$ as in \eqref{def:limsupa1}
 \begin{equation}\label{est:zbar-bound-2}
\limsup\limits_{t\to\infty}\int_U| z |^\ah\,dx\leq  C\left(1+A_1(\ah)\right).
\end{equation}
Consequently, if $Q(t)$ and $\Psi(x,t)$ satisfy condition $A1$ then 
\begin{equation}\label{est:zbar-bound-1}
\limsup\limits_{t\to\infty}\int_U| z |^\ah\,dx\leq C.
\end{equation}
\end{theorem}
\begin{proof}
 The result follows from Lemma~\ref{lem:A1-gronw}.
\end{proof}


\vspace{0.3cm}

\subsection{Bounds for the gradient of solutions}
We will obtain the bounds for the integral $\int_U H(|\n p|)\,dx$, where $H(|\n p|)$ is defined in \eqref{h-def}. This result is necessary for our estimates of the difference in time derivative of fully transient and PSS solutions. 
%

\begin{theorem}\label{thrm:bdd-h-gradp}
Let $A_2(t)$ be as in \eqref{def:a2}. For $\ah\geq na/(2-a)$
\begin{equation}\label{ip-limsup}
\limsup\limits_{t\to\infty}\int_U H(|\n p|)\,dx\leq 
C\left(1+A_1(\ah)+\limsup\limits_{t\to\infty}A_{2}(t)\right).
\end{equation}
Consequently,  if $Q(t)$ and $\Psi(x,t)$ satisfy conditions $A2$ then 
 \begin{equation}\label{ip-bounded}
\int_U|\nabla p|^{2-a}\,dx\leq  C\left(\int_U H(|\n p|)\,dx+1\right)\leq C.
 \end{equation}
\end{theorem}

\begin{proof}
 In formula (4.46) in \cite{ABI12} we have for $z=q-q_s$:
\begin{align}\label{qt-qst-6}
\int_U q_t^2\,dx&+\frac{d}{dt}\left[I_1[p](t)+\int_Uz^2\,dx\right]\
\leq-C\,I_1[p](t)+C(1+A_{2}(t)),
\end{align}
where
\begin{equation}\label{i1p}
 I_1[p](t)=\int_U H(|\n p|)\,dx-\int_UK(|\nabla W|)\nabla W\nabla p\,dx+\Delta_Q(t)\Delta_\gamma(t).
\end{equation}
Notice that
\begin{align}\label{eq:ddt-qqs-equal}
 \frac{d}{dt}\int_Uz^2\,dx=2\int_U zz_t\,dx
\geq-\int_Uz^2\,dx-\int_U q_t^2\,dx.
\end{align}
Subtracting \eqref{eq:ddt-qqs-equal} from \eqref{qt-qst-6} we get 
\begin{align}\label{qt-qst-7}
\frac{d}{dt}\, I_1[p](t)\leq-C\,I_1[p](t)+\int_Uz^2\,dx+C(1+A_{2}(t)).
\end{align}
The term $\int_Uz^2\,dx$ can be estimated using  \eqref{est:zbar-bound-2} as $\ah>2$. Then in \eqref{qt-qst-7} we have
 \begin{align}\label{qt-qst-8}
\frac{d}{dt}\, I_1[p](t)\leq-C\,I_1[p](t)
+C\left[1+A_{2}(t)+A_1(\ah)\right].
\end{align}

\noindent Applying Gronwall's inequality one has for $t\geq0$
\begin{align}\label{qt-qst-9}
I_1[p](t)\leq
e^{-c_1t}\left[I_1[p](0)
+C\int_0^t e^{c_1\tau}\left(1+A_{2}(\tau)+
A_1(\ah)\right)\,d\tau\right]\notag\\
\leq C(1+A_1(\ah))+e^{-c_1t}I_1[p](0)
+C\int_0^t e^{-c_1(t-\tau)}A_{2}(\tau)\,d\tau.
\end{align}
According to formula (2.35) in \cite{HI12anydegree}
\begin{equation}\label{est:limsup-int-0t}
 \limsup\limits_{t\to\infty} \int_0^te^{-k(t-\tau)}f(\tau)\,d\tau\leq \limsup\limits_{t\to\infty}k^{-1}f(t).
\end{equation}
Then using estimates \eqref{est:limsup-int-0t} and \eqref{est:zbar-bound-2} in \eqref{qt-qst-8} one has
\begin{align}\label{qt-qst-10}
\limsup\limits_{t\to\infty}I_1[p](t)\leq 
C\left(1+A_1(\ah)+\limsup\limits_{t\to\infty}A_{2}(t)\right).
\end{align}

\noindent
Under assumptions $A2$ the terms $\limsup\limits_{t\to\infty}A_{2}(t)$ and $A_1(\ah)$ are bounded. \\
Thus $\limsup\limits_{t\to\infty}I_1[p](t)\leq C$. Then from formula \eqref{i1p} for $I_1[p]$ it follows:
\begin{equation}\label{estim-H-final}
 \int_UH(|\n p|)\,dx\leq \int_UK(|\nabla W|)|\nabla W||\nabla p|\,dx+|\Delta_Q(t)\Delta_\gamma(t)|+C.
\end{equation}
For the first integral in the RHS of \eqref{estim-H-final} we have
\begin{align}\label{ddt-q-qs-5}
 \int_UK(|\nabla W|)|\nabla W||\nabla p|\,dx\leq \e_1\,\int_UH(|\n p|)\,dx+C\|\nabla W\|_{L^{2-a}(U)}+C.
\end{align}
According to Lemma 4.4 in \cite{ABI12} with $F_2(t)$ as in \eqref{def:f2} we have
\begin{equation}\label{dgamma-dq}
 |\Delta_\gamma(t)\Delta_Q(t)|\leq \e_2\,\int_UH(|\n p|)\,dx + C F_2(t)|\Delta_Q(t)|.
\end{equation}
Choosing $\e_1$ and $\e_2$ to be small enough and using \eqref{estim-H-final} in \eqref{qt-qst-10} we obtain \eqref{ip-limsup}. Since $F_2(t)$ is uniformly bounded under assumptions A2, the 
estimate \eqref{ip-bounded} follows.

\end{proof}

\subsection{Estimate of time derivatives}
We will obtain the estimate on the difference of time derivative of fully transient and PSS solutions.

Let $\pbar(x,t)=p(x,t)-\Psi(x,t)$ and $\pbar_s(x,t)=p_s(x,t)-\Phi(x)$, then
\begin{align}
 &\frac{\partial \pbar}{\partial t}=L[p]-\Psi_t(x,t), \quad \pbar|_{\Gamma_i}=\gamma(t);\label{eq:trans-shift}\\
 &\frac{\partial \pbar_s}{\partial t}=-A=L[W], \quad \pbar_s|_{\Gamma_i}=-At.\label{eq:pss-shift}
\end{align}


\begin{theorem}\label{thr:pt-bound-ah}
Suppose that boundary data satisfies Assumptions $A3$-$\beta$.
Then for any $\beta\geq2$
\begin{equation}
\int_U|\pbar_t-\pbar_{s,t}|^\beta\,dx= \int_U|\pbar_t+A|^\beta\,dx\leq C.
\end{equation}
\end{theorem}
\begin{proof}
 Let $\zbar_t=\pbar_t+A$ and notice that $\n\zbar_t=\n\pbar_t$. Subtracting \eqref{eq:trans-shift} and \eqref{eq:pss-shift}, taking derivative in $t$, multiplying on 
$|\zbar_t|^{\beta-1}{\rm sign}(\zbar_t)$, integrating over $U$ and integrating by parts we get
\begin{align}
 \frac1\beta\frac{d}{dt}&\int_U|\zbar_t|^\beta\,dx\label{ddt-pt-1-ah}
 =\int_U(L[p]-\Psi_t-L[W])_t|\zbar_t|^{\beta-1}{\rm sign}(\zbar_t)\,dx \\
=&\int_U(\nabla\cdot (K(|\nabla p|)\nabla p))_t|\zbar_t|^{\beta-1}{\rm sign}(\zbar_t)\,dx
-\int_U\Psi_{tt}|\zbar_t|^{\beta-1}{\rm sign}(\zbar_t)\,dx\notag\\
=&-(\beta-1)\int_U(K(|\nabla p|)\nabla p)_t\cdot \nabla\pbar_t\ |\zbar_t|^{\beta-2}\,dx
-\int_U\Psi_{tt}|\zbar_t|^{\beta-1}{\rm sign}(\zbar_t)\,dx\notag\\
&+\int_{\Gamma_i}(K(|\nabla p|)\nabla p)_t\cdot N |\zbar_t|^{\beta-1}{\rm sign}(\zbar_t)\,ds.\notag
\end{align}
Since $ (\pbar_t+A) |_{\Gamma_i}=\D_\gamma'(t)$ the boundary integral is equal to $(\D_\gamma'(t))^{\beta-1}Q'(t)$.

\noindent The quantity in the first integral in the RHS of \eqref{ddt-pt-1-ah} is equal to
\begin{align*}
 -(K&(|\nabla p|)\nabla p)_t\cdot \nabla\pbar_t|\zbar_t|^{\beta-2}\\
&=-K(|\nabla p|)\nabla p_t\cdot \nabla\pbar_t|\zbar_t|^{\beta-2}-
K'(|\nabla p|)\frac{(\nabla p_t\cdot \nabla p)(\nabla p\cdot\nabla\pbar_t)}{|\nabla p|}|\zbar_t|^{\beta-2}.\notag
\end{align*}
From \eqref{k-prime} it follows that 
\begin{align*}
 K'(|\nabla p|)&\frac{(\nabla p_t\cdot \nabla p)(\nabla p\cdot\nabla\pbar_t)}{|\nabla p|}|\zbar_t|^{\beta-2}
\leq a K(|\nabla p|)|\nabla p_t||\nabla\pbar_t||\zbar_t|^{\beta-2} \\
&\leq a K(|\nabla p|)|\nabla\pbar_t|^2|\zbar_t|^{\beta-2}+a K(|\nabla p|)
|\nabla\Psi_t||\nabla\pbar_t||\zbar_t|^{\beta-2}.\notag
\end{align*}
Then in \eqref{ddt-pt-1-ah} we have
\begin{align}\label{ddt-pt-1a-ah-1}
\frac1\beta\frac{d}{dt}\int_U|\zbar_t|^\beta\,dx
\leq& -(1-a)(\beta-1)\int_UK(|\nabla p|)|\nabla\pbar_t|^2|\zbar_t|^{\beta-2}\,dx\notag\\
&+(1+a)(\beta-1)\int_UK(|\nabla p|)|\nabla\Psi_t||\nabla\pbar_t||\zbar_t|^{\beta-2}\,dx\notag\\
&+\int_U|\zbar_t|^{\beta-1}|\Psi_{tt}|\,dx+|(\Delta'_{\gamma}(t))^{\beta-1}Q'(t)|.
\end{align}

\noindent To estimate the second term in the RHS of \eqref{ddt-pt-1a-ah-1} we apply H\"{o}lder inequality and Young's inequality with $\e=(1+a)/(1-a)$. Since $K(|\nabla p|)$ is bounded (see \eqref{k-ineq}) applying once more H\"{o}lder inequality with powers $\beta/(\beta-2)$ and $\beta/2$ we get
\begin{align}\label{ineq_1}
\int_UK(|\nabla p|)|\nabla\Psi_t||\nabla\pbar_t||\zbar_t|^{\beta-2}\,dx\leq\,\,&\frac{1-a}{2(1+a)}\int_UK(|\nabla p|)|\nabla\pbar_t|^2|\zbar_t|^{\beta-2}\,dx\notag\\
&+C\|\nabla\Psi_t\|_{L^\beta(U)}^2\left(\int_U|\zbar_t|^{\beta}\,dx\right)^{\frac{\beta-2}{\beta}}.
\end{align}
Similar for the third term in the RHS of \eqref{ddt-pt-1a-ah-1} one has
\begin{align}\label{ineq_2}
\int_U|\zbar_t|^{\beta-1}|\Psi_{tt}|\,dx \leq C\|\Psi_{tt}\|_{L^{\beta}(U)}\left(\int_U|\zbar_t|^{\beta}\,dx+1\right).
\end{align}

Then in \eqref{ddt-pt-1a-ah-1} we have
\begin{align}\label{ddt-pt-1a-ah}
\frac1\beta\frac{d}{dt}\int_U|\zbar_t|^\beta\,dx
\leq& -C\int_UK(|\nabla p|)|\nabla\pbar_t|^2|\zbar_t|^{\beta-2}\,dx\notag\\
&+C\|\nabla\Psi_t\|_{L^\beta(U)}^2\left(\int_U|\zbar_t|^{\beta}\,dx+1\right)\notag\\
&+C\|\Psi_{tt}\|_{L^{\beta}(U)}\left(\int_U|\zbar_t|^{\beta}\,dx+1\right)+|(\Delta'_{\gamma}(t))^{\beta-1}Q'(t)|.
\end{align}

Since the first term in the RHS of \eqref{ddt-pt-1a-ah} is negative, it  can be neglected. We then have
\begin{equation}\label{z_1}
\frac{d}{dt} \int_U|\zbar_t|^\beta\,dx\leq f_1(t)\int_U|\zbar_t|^\beta\,dx+ f_2(t),
\end{equation}
where
\begin{align*}
&f_1(t)=C\|\Psi_{tt}\|_{L^{\beta}(U)}+C\|\nabla\Psi_t\|_{L^{\beta}(U)}^2,\\
&f_2(t)=f_1(t)+|(\Delta'_{\gamma}(t))^{\beta-1}Q'(t)|.
\end{align*}
The result follows from the Gronwall's inequality under assumptions A3-$\beta$.

\end{proof}

In order to prove that $\limsup\limits_{t\to\infty}\int_U|\pbar_t+A|^2\,dx= 0$ we will need the Weighted Poincar\'{e} inequality (see Lemma 2.5., \cite{HI12anydegree}):
\begin{lemma}
Let $\xi=\xi(x)\geq0$ and function $u(x)$ be defined on $U$ and is vanishing on the boundary  $\partial U$. Assume $\ah\geq2$. Given two numbers
$\theta$ and $\theta_1$ such that
\begin{equation*}
 \theta>\frac2{(2-a)^*}\quad \text{and}\quad \max\left\{1,\frac{2n}{n\theta+2}\right\}\leq\theta_1<2-a,
\end{equation*}
there is constant $C$ such that
\begin{equation}\label{poincare-w-0}
 \int_U |u|^\alpha\,dx\leq CM_P(\xi,|u|)\left[\int_U K(\xi)|\nabla u|^2|u|^{\alpha-2}\,dx\right]^{\frac1\theta},
\end{equation}
where 
\begin{align}
 &M_P(\xi,|u|)=\left[1+\int_U\xi^{2-a}+|u|^{\theta_2\alpha}\,dx\right]^{\frac{2-\theta_1}{\theta\theta_1}}, \label{def:poinc-Mxiu}\\
&\theta_2=\frac{\theta_1(\theta-1)(2-a)}{2(2-a-\theta_1)}>0.\notag
\end{align}
\end{lemma}


\vspace{0.3cm}
\begin{theorem}\label{thr:pt+A-to0}
If $Q(t)$ and $\Psi(x,t)$ satisfy assumptions $A3$-$\beta$ and $A4$, then
\begin{equation*}
\limsup\limits_{t\to\infty}\int_U|\pbar_t-\pbar_{s,t}|^2\,dx= 
\limsup\limits_{t\to\infty}\int_U|\pbar_t+A|^2\,dx= 0.
\end{equation*}
\end{theorem}
\begin{proof}
In (5.15) of \cite{ABI12} we have 
 \begin{align}\label{ddt-pt-3}
\frac{d}{dt}\int_U&(\pbar_t+A)^2\,dx\leq-C\int_UK(|\nabla p|)|\nabla\pbar_t|^2dx+\e_2\int_U|\pbar_t+A|^2\,dx
+C A_3(\e_2,t),
\end{align}
where
the constant $C$ depends on $\int_UH(|\n p|)\,dx$ and is bounded under assumptions $A2$.

We will use Weighted Poincar\'{e} inequality \eqref{poincare-w-0} to estimate $\int_UK(|\nabla p|)|\nabla\pbar_t|^2dx$ in terms of $\int_U|\pbar_t+A|^2\,dx$.
Namely, let $u=\pbar_t+A-\D'_\g(t)$ and  $\xi(x)=|\nabla p|$. Then $u|_{\Gamma_i}=0$ and $\nabla u=\nabla\pbar_t$. 
Then we have 
\begin{align}\label{pta-0-1}
 \int_U|\pbar_t+A|^2\,dx&=\int_U|u+\D'_\g(t)|^2\,dx\leq \frac32\int_U|u|^2\,dx+\frac32\int_U|\D'_\g(t)|^2\,dx\\
 &\leq C M_P(\xi,|u|)\cdot\left[\int_U K(\nabla p)|\nabla \pbar_t|^2\,dx\right]^{\frac1\theta}+\frac{|U|}2|\D'_\g(t)|^2.\notag
\end{align}
Here 
\begin{align*}
 M_P(|\nabla p|,|u|)=\left[1+\int_U|\nabla p|^{2-a}+|\pbar_t+A-\D'_\g(t)|^{2\theta_2}\,dx\right]^{\frac{2-\theta_1}{\theta\theta_1}}\leq C
\end{align*}
in view of Theorem~\ref{thrm:bdd-h-gradp} under Assumptions A2 and Theorem~\ref{thr:pt-bound-ah} under Assumptions $A3$-$\beta$.

Thus
\begin{align}
 \int_U|\pbar_t+A|^2\,dx\leq C \left[\int_U K(\nabla p)|\nabla \pbar_t|^2\,dx\right]^{\frac1\theta}+\frac{|U|}2|\D'_\g(t)|^2.\notag
\end{align}
On the other hand  note that $\int_U q_t\,dx=0$ and $q_t|_{\Gamma_i}=(q_t-q_{s,t})|_{\Gamma_i}=\Delta'_\gamma(t)+F'_1(t)$. Here
$F'_1(t)=\D_Q(t)+\int_U\D\Psi_t\,dx$ (see \eqref{def:f1}). Notice that $\limsup\limits_{t\to\infty}|F'_1(t)|=0$ under conditions A4.
Then by Trace theorem, Poincar\'{e} inequality and the fact that $\int_U q_t dx=0$ we have
\begin{equation}
 |\Delta'_{\gamma}(t)|^2\leq C\left(\int_{\Gamma_i}|q_t|\,ds+C|F'_1(t)|\right)^2\leq C \left(\int_U|\nabla q_t|\,dx\right)^2+C|F'_1(t)|^2.\notag
\end{equation}

Next, as $\nabla q_t= \nabla\pbar_t$, applying H\"{o}lder inequality we have 
\begin{equation}
 \left(\int_U|\nabla q_t|\,dx\right)^2 \leq C \int_UK(|\nabla p|)|\nabla\pbar_t|^2\,dx\cdot
\int_UK(|\nabla p|)^{-1}\,dx.
\end{equation}
Notice that from \eqref{k-ineq} and H\"{o}lder inequality it follows
\begin{align}\label{k--1}
 \int_UK(|\nabla p|)^{-1}\,dx&\leq C\int_U(1+|\nabla p|)^{a}\,dx\leq  C + C\int_U|\nabla p|^{2-a}\,dx\leq C,
\end{align}
The last estimate in \eqref{k--1} follows from  Theorem~\ref{thrm:bdd-h-gradp} under conditions A2.\\
Thus 
\begin{equation}
 |\Delta'_{\gamma}(t)|^2\leq C\int_UK(|\nabla p|)|\nabla\pbar_t|^2\,dx+C|F'_1(t)|^2.\notag
 \end{equation}
Then in \eqref{pta-0-1} we have
\begin{align}\label{pta-0-2}
 \int_U|\pbar_t+A|^2\,dx
 \leq  C\,\left(I^{\frac1\theta}(t)+I(t)\right)+C|F'_1(t)|^2=C\cdot \varphi\Big(I(t)+|F'_1(t)|^2\Big),
\end{align}
where
\begin{equation}
 \varphi(s)=s^{\frac1{\theta}}+s+|F'_1(t)|^2;\quad I(t)=\int_U K(|\nabla p|)|\nabla \pbar_t|^2\,dx.
\end{equation}
Due to condition on $|F'_1(t)|$  
\begin{equation*}
 I(t)\geq C\varphi^{-1}\left(\int_U|\pbar_t+A|^2\,dx\right)-C|F'_1(t)|^2.
\end{equation*}

Then due to definition \eqref{def:a3} of $A_3(\e_2,t)$ one has in \eqref{ddt-pt-3} 
\begin{align}\label{ddt-pt-4-new}
\frac{d}{dt}\int_U(\pbar_t+A)^2\,dx\leq&-C\varphi^{-1}\left(\int_U|\pbar_t+A|^2\,dx\right)
+\e_2\int_U|\pbar_t+A|^2\,dx+CA_3(\e_2,t).
\end{align}

Let $y(t)=\int_U|\pbar_t+A|^2\,dx$. Due to Theorem \eqref{thr:pt-bound-ah}  one has in \eqref{ddt-pt-4-new}
\begin{align}
y'(t)\leq -C f(y(t))+CA_3(\e_2,t)+C\e_2,
\end{align}
where
\begin{align}
f(s)=\varphi^{-1}(s),\ f(s)>0 \quad \text{for} \quad s>0,\quad \text{and}\quad  f(0)=0.
\end{align}
Under the Assumptions $A3$-$\beta$ we have $0\leq y(t)\leq C$. According to Lemma~\ref{lem:A1-gronw} we will get 
$\limsup_{t\to\infty}y(t)\le C\varphi\big(\limsup_{t\to\infty}(A_3(\e_2,t)+\e_2\big)$. Under Assumptions A4
$\limsup_{t\to\infty} A_3(\e_2,t)=0$ and therefore 
$$\limsup_{t\to\infty}y(t)\le \varphi(\e_2).$$
Taking in above $\e_2\to 0$ we complete   proof of the Theorem  \ref{thr:pt+A-to0}. 

\end{proof}

Finally Theorems~\ref{th:zbar-to0}, \ref{thrm:bdd-h-gradp} and \ref{thr:pt+A-to0} and Lemma~4.1, \cite{ABI12}, under the conditions A1-A5 allow us to conclude that $\|\nabla(p-p_s)\|_{L^{2-a}}\to 0$ as $t\to\infty$ in Lemma~\ref{lem:npps0}. Combining  it with Lemma~\ref{lem:pi-diff} we get the main result for the solutions of IBVP-I: 
\begin{theorem}\label{th:pi-pipss-0}
 If boundary data $Q(t)$ and $\Psi(x,t)$ satisfy assumptions $A5$
then 
\begin{equation}
 J_g(t)-J_{g,PSS}\to 0 \quad\text{as}\quad t\to\infty.
\end{equation}
\end{theorem}

\vspace{0.3cm}

\section{IBVP-II: Asymptotic convergence of Diffusive Capacity}\label{sec:dirichlet}

Let $p(x,t)$ be the solution of IBVP-II \eqref{eq:p}, \eqref{bc:0-neuman}, \eqref{initialcond} with Dirichlet boundary condition \eqref{bc:dir} with boundary data $\Psi(x,t)$. Let $J_g(t)$ be the corresponding diffusive capacity. 
Let $p_s(x,t)=-At+W(x)$, $W(x)|_{\Gamma_i}=\varphi(x)$, be the PSS solution of this IBVP defined by the boundary data $-At$, $A=Q_s/|U|$, and $\Phi(x)$. Let $J_{g,PSS}$ be the corresponding diffusive capacity.

\subsection{Assumptions on the boundary data}
Let
\begin{align}
D(\beta,t)=\|\Psi_{tt}\|_{L^{\beta}(U)}+\|\nabla\Psi_t\|_{L^{\beta}(U)}^2 \quad\text{for any $\beta\geq2$}.\label{dir:d1}
\end{align}

We also will use some notations from \cite{HI12anydegree}.
 For $\ah\geq\frac{na}{2-a}$, $b=\frac{\ah(2-a)}{2}$, $r_0=\frac{n(2-a)}{(2-a)(n+1)-n}$, let
\begin{align} 
&A(\ah,t)=
\|\n\Psi\|_{L^b}^{\ah-a}
+\left[\int_U|\Psi_t|^{\ah}\,dx\right]^{\frac{\ah-a}{\ah(1-a)}};\\
 &G_4(t)= \int_U|\n\Psi|^2\,dx+\left[\int_U|\Psi_t|^{r_0}\,dx\right]^{\frac{2-a}{r_0(1-a)}}
+\|\Psi_t\|_{L^{r_0}}\notag\\
&\hspace{2.cm}+\int_U|\n\Psi_t|^2\,dx+\int_U|\Psi_t|^2\,dx+\int_U|\Psi_{tt}|^2\,dx.
\end{align}

 The results in this section are obtained under the following conditions on the boundary data 

 {\bf Assumptions D1:}
\begin{equation}\label{assumpD1}
\int_{1}^{\infty}D(\beta,t)\,dt<\infty\quad\text{for any $\beta\geq2$}.
\end{equation}

{\bf Assumptions D2:}
\begin{align}
& \limsup\limits_{t\to\infty}\left(A(\ah,t)+G_4(t)\right)\leq C;\label{assumpD2-1}\\
 & \limsup_{t\to\infty}\left(D(2,t)+\int_U|\Psi_t+A|^2\,dx\right)=0. \label{assumpD2}
\end{align}

{\bf Assumptions D3:}
\begin{align} 
& \limsup\limits_{t\to\infty}A(\ah,t)+\|\n\Phi(x)\|_{L^b}^{\ah-a}\leq C; \label{assumpD3-1}\\
& \limsup\limits_{t\to\infty}\|\nabla(\Delta_\Psi)\|_{L^{2}(U)}=0.\label{assumpD3-2}
\end{align}

\vspace{0.2cm}


According to Lemma~\ref{lem:pi-diff} in order to prove $J_g(t)\to J_{g,PSS}$ as $t\to0$ 
we need to prove the following estimates: 
 \begin{align}
  &\Delta_Q(t)\to 0\quad\text{ as}\quad t\to\infty\quad\text{see Lemma~\ref{lem:deltaq-0};}\\
  &\int_U|\nabla(p-p_s)|^{2-a}\,dx\to0\quad\text{ as}\quad t\to\infty\quad\text{seeTheorem~\ref{th:dir:p-ps}.}
 \end{align}

\subsection{Estimates on the solution of IBVP-II}
Along with the solutions $p$ and $p_s$ we will use the shifts  $\pt(x,t)=p(x,t)-\Psi(x,t)$ and $\pt_s(x,t)=p_s(x,t)-\Phi(x)+At$. These shifts Obvioulsly satisfy
\begin{align}
 &\frac{\partial \pt}{\partial t}=L[p]-\Psi_t(x,t), \quad 
\pt|_{\Gamma}=0;\label{eq:dir-trans-shift}\\
 &\frac{\partial \pt_s}{\partial t}=0=L[W]+A, \quad \pt_s|_{\Gamma}=0.\label{eq:dir-pss-shift}
\end{align}

\begin{lemma}\label{thrm:pt-ah-bound-dir}
 Suppose boundary data $\Psi(x,t)$ satisfies assumptions $D1$, 
then 
\begin{equation}
 \int_U|\pt_t|^\beta\,dx\leq C\quad\text{for any $\beta\geq2$.}
\end{equation}
\end{lemma}
\begin{proof}
 Similar to \eqref{ddt-pt-1-ah} we have 
 \begin{align}\label{ddt-pt-dir}
\frac1\beta\frac{d}{dt}\int_U|\pt_t|^\beta\,dx
\leq& -(1-a)(\beta-1)\int_UK(|\nabla p|)|\nabla\pt_t|^2|\pt_t|^{\beta-2}\,dx\notag\\
&+(1+a)(\beta-1)\int_UK(|\nabla p|)|\nabla\Psi_t||\nabla\pt_t||\pt_t|^{\beta-2}\,dx\notag\\
&+\int_U|\pt_t|^{\beta-1}|\Psi_{tt}|\,dx.
\end{align}
Similar to \eqref{ineq_1}  - \eqref{z_1} we have
\begin{equation}
\frac{d}{dt}\int_U|\pt_t|^\beta\,dx\leq D(\beta,t)\int_U|\pt_t|^\beta\,dx+D(\beta,t).
\end{equation}
The result follows from the Gronwall's inequality under assumptions D1.
\end{proof}



\vspace{0.5cm}
Further we will need the following result obtained in \cite{HI12anydegree}:
under assumptions D2.1 \eqref{assumpD2-1} (see \cite{HI12anydegree}, Th. 4.5)
\begin{equation}
 \int_U|\n p|^{2-a}\,dx\leq C\int_UH(|\n p|)\,dx+C\leq C. \label{dir:int_h}
\end{equation}


\begin{lemma}\label{lem:deltaq-0}
Suppose boundary data $\Psi(x,t)$ satisfies assumptions $D1$ and $D2$. Then 
 \begin{align}\label{est:pta2}
  \limsup\limits_{t\to\infty}\left(|\Delta_Q(t)| + C\int_U|p_t+A|^2\,dx\right)=0 \quad\text{as}\quad t\to\infty. 
 \end{align}
\end{lemma}
\begin{proof}
First, notice that
\begin{equation*}
 \frac{d}{dt}\int_Up(x,t)\,dx=Q(t) \quad\text{and}\quad  \frac{d}{dt}\int_Up_s(x,t)\,dx=Q_s.
\end{equation*}
Thus 
 \begin{align}\label{delta-q2}
 |\Delta_Q(t)|^2&= |Q(t)-Q_s|^2=\left|\int_Up_t\,dx-Q_s\right|^2=\left|\int_U(p_t+A)dx\right|^2\notag\\
&\leq C\int_U|p_t+A|^2\,dx\leq C\int_U|\pt_t|^2\,dx +C\int_U|\Psi_t+A|^2\,dx.
 \end{align}
Assuming $\int_U|\Psi_t+A|^2\,dx\to 0$ as $t\to\infty$
it is sufficient to prove that $\int_U|\pt_t|^2\,dx\to0$ as $t\to\infty$.

Differntiating both sides of \eqref{eq:dir-trans-shift} in $t$, multiplying by $\pt_t$ and integrating over $U$ we have due to zero boundary condition 
\begin{align}\label{pbart-a-1}
  \frac12\frac{d}{dt}\int_U\pt_t^2\,dx=-\int_U(K(|\nabla p|)\nabla p)_t\nabla\pt_t\,dx
-\int_U\Psi_{tt}\,\pt_t\,dx.
\end{align}
Similar to estimate (5.10) in \cite{ABI12} we get since $\nabla\pt=\nabla p-\nabla\Psi$
\begin{align}\label{pbart-a-2}
\frac{d}{dt}\int_U\pt_t^2\,dx\leq-C\int_UK(|\nabla p|)|\nabla\pt_t|^2\,dx
+\e\left(\int_U|\pt_t|^2\,dx\right)^{\frac2{2-a}}+C\widetilde{D}(2,t).
\end{align}
where
\begin{equation}
 \widetilde{D}(2,t)=\|\Psi_{tt}\|_{L^{2}}^{\frac4{2+a}}+\|\nabla\Psi_t\|_{L^{2}}^2.
\end{equation}

\noindent Applying the Weighted Poincar\'{e} inequality \eqref{poincare-w-0} with $u=\pt_t$, $\pt_t|_{\Gamma_i}=0$ and $\xi=\n p$ with powers $\ah=2$ and $\theta=\frac2{2-a}$ to the first integral on the RHS of \eqref{pbart-a-2} we get: 
\begin{align*}
\frac{d}{dt}\int_U\pt_t^2\,dx\leq-CM_P^{-\theta}\left(\int_U|\pt_t|^2\,dx\right)^\frac2{2-a}
+\e\left(\int_U|\pt_t|^2\,dx\right)^\frac2{2-a}+C\widetilde{D}(2,t).
\end{align*}
Here
\begin{equation}
 M_P(|\nabla p|,|\pt_t|)=\left[1+\int_U|\nabla p|^{2-a}+|\pt_t|^{2\theta_2}\,dx\right]^{\frac{2-\theta_1}{\theta\theta_1}}\leq C
\end{equation}
by virtue of \eqref{dir:int_h} under assumptions D2.1 and Lemma~\ref{thrm:pt-ah-bound-dir} under assumptions D1.

Choosing sufficiently small $\e$ we get 
\begin{equation}\label{pta5}
 \limsup\limits_{t\to\infty}\int_U|\pt_t|^2\,dx\leq \limsup\limits_{t\to\infty} \widetilde{D}(2,t).
\end{equation}
Under assumptions D2.2 the RHS of \eqref{pta5} converges to 0 as $t\to\infty$ and the result of Lemma follows from \eqref{delta-q2}.
\end{proof}



Further we will use another result from \cite{HI12anydegree}:
under assumptions D3.1 \eqref{assumpD3-1} (see \cite{HI12anydegree}, Th. 4.3)
\begin{equation}
 \int_U|\pt(x,t)|^2+|\pt_s(x,t)|^2\,dx\leq C.\label{dir:pbar-psbar}
\end{equation}

\begin{theorem}\label{th:dir:p-ps}
Suppose the boundary data $\Psi(x,t)$ and $\Phi(x)$ satisfies assumptions $D1 - D3$. Then
\begin{equation}
 \|\nabla(p-p_s)\|^{2}_{L^{2-a}(U)}\to 0 \quad\text{as}\quad t\to\infty.
\end{equation}

\end{theorem}
\begin{proof}
From \eqref{ineq:monoton} we get
\begin{equation}\label{phi-2-a}
\left(\int_{U}|\nabla(p-p_s)|^{2-a}dx\right)^{\frac2{2-a}}\leq 
C_\Phi\int_U(K(|\nabla p|)\nabla p - K(|\nabla p_s|)\nabla p_s)\cdot \nabla(p-p_s)dx,
\end{equation}
where $C_\Phi=C_\Phi(p,p_s)$ is as in \eqref{def:c-phi}. Since $\nabla(p-p_s)=\n(\pt-W)+\nabla(\Delta_\Psi)$, $C_\Phi$ is bounded by virtue of \eqref{dir:int_h} and \eqref{ineq:nps-bound} under condition $\|\nabla(\Delta_\Psi)\|_{L^{2-a}(U)}\leq C$.  

Since $\pt-\pt_s|_{\Gamma_i}=0$, integration by parts in the  RHS of \eqref{phi-2-a} gives
\begin{align*}
\|\nabla(p-p_s)\|^{2}_{L^{2-a}}\leq&-C\int_U(p_t-p_{s,t})(\pt-\pt_s)\,dx\\
&+C\int_U(K(|\nabla p|)\nabla p - K(|\nabla p_s|)\nabla p_s)\cdot \nabla(\Delta_\Psi)dx.\notag
\end{align*}

\noindent Since $K(\cdot)$ is bounded we have
\begin{align}
 \int_U&(K(|\nabla p|)\nabla p - K(|\nabla p_s|)\nabla p_s)\nabla(\Delta_\Psi)dx
\\
&\leq\left(\int_UK(|\nabla p|)|\nabla p|^2\,dx\right)^{\frac12}\left(\int_UK(|\nabla p|)|\nabla(\Delta_\Psi)|^2\,dx\right)^{\frac12}\notag\\
&+
\left(\int_UK(|\nabla W|)|\nabla W|^2\,dx\right)^{\frac12}\left(\int_UK(|\nabla W|)|\nabla(\Delta_\Psi)|^2\,dx\right)^{\frac12}\notag\\
&\leq 
\left[\left(\int_UH(|\n p|)\,dx\right)^{\frac12}+
\left(\int_U|\n W|^{2-a}\,dx\right)^{\frac12}\right]\|\nabla(\Delta_\Psi)\|_{L^{2}(U)}\notag\\
&\leq C \|\nabla(\Delta_\Psi)\|_{L^{2}(U)}.\notag
\end{align}
The last inequality holds in view of \eqref{dir:int_h} and \eqref{ineq:nps-bound}.

\noindent Thus 
\begin{align*}
\|\nabla(p-p_s)\|^{2}_{L^{2-a}}\leq&-C\int_U(p_t+A)(\pt-\pt_s)\,dx+C \|\nabla(\Delta_\Psi)\|_{L^{2}(U)}.
\end{align*}

\noindent The result follows in view of Lemma~\ref{lem:deltaq-0} and estimate \eqref{dir:pbar-psbar} under the assumptions D1 - D3.

\end{proof}

Lemma~\ref{lem:deltaq-0} and Theorem~\ref{th:dir:p-ps} combined with Lemma~\ref{lem:pi-diff} prove the main result 
for the solution of IBVP-II:
\begin{theorem}\label{th:pi-pipss-dir}
Suppose the boundary data satisfies assumptions $D1 - D3$. Then
\begin{equation*}
 J_g(t)-J_{g,PSS}\to0\quad\text{as}\quad t\to\infty.
\end{equation*}
\end{theorem}

\section{Productivity index concept for the flow of ideal gas.}\label{sec:Gas}
In this section we discuss the concept of the Diffusive Capacity 
for an ideal gas flow in the porous media 
for a well-reservoir system.  We will repeat some discussion from Sec.~\ref{sec:problem-state} and provide more details. 
The equation of state for an ideal gas takes the form  (see \cite{Aronson}, \cite{Muskat})
\begin{equation}\label{eq:state-gas}
 \rho(p)=Mp,
\end{equation}
where $\rho$ is the density of the fluid and $M$ is a proportionality constant. 
Without loss of generality we assume $M=1$.
As a momentum equation we consider the second order Forchheimer equation \eqref{eq:g-forch} 
which takes the form (see, for example,  \cite{PayneStraughan-cont-conv})                                      
\begin{align}\label{eq:2forch-1}
 \alpha u + \beta\rho u|u|=-\nabla p,\quad \beta=F\Phi k^{-1/2},
\end{align}
where $F$ is the Forchheimer coefficient, 
$\Phi$ is porosity and $k$ is permeability of the porous media.

As in case of slightly compressible fluid, the above equation can be solved for $u$
and rewritten in terms of a nonlinear permeability function $K_2(|\n p|)$:
\begin{equation}\label{eq:2forch}
 u=-K_2(p|\n p|)\n p=-\frac{2}{\alpha+\sqrt{\alpha^2+4\beta\rho|\nabla p|}}\nabla p.
\end{equation}
Combining \eqref{eq:state-gas} and \eqref{eq:2forch} together with the continuity equation \eqref{eq:continuity} yields
the parabolic equation for pressure
\begin{equation}\label{gas-equation}
 \frac{\partial p}{\partial t}=L[p]\equiv\nabla\cdot( K_2(p|\n p|)p\n p)=
 \nabla\cdot\left(\frac{2 p }{\alpha+\sqrt{\alpha^2+4\beta p|\nabla p|}}\nabla p\right).
\end{equation}
As before the domain $U$ models the reservoir with boundary split in two parts
$\Gamma_i$ and $\Gamma_e$. $\Gamma_i$ is the well-boundary 
while $\Gamma_e$ is an exterior non-permeable boundary of the reservoir.\\
We impose  Dirichlet Pseudo Steady State (PSS)  
boundary conditions on $\Gamma_i$
\begin{equation}
 p|_{\Gamma_i}=B-At\label{bc-gas:dirichlet}
\end{equation}
and zero mass flux boundary conditions  on $\Gamma_e$
\begin{equation}
 \rho u\cdot N|_{\Gamma_e}=0 \label{bc-gas:neum}
\end{equation}
and the initial condition
\begin{equation}
 p(x,0)=\sqrt{B^2+\phi_0(x)}\label{gas-init-cond}
\end{equation}
 for $\phi_0(x)\ge0$.
\begin{remark}
 Constant $B$ is a  
positive parameter (generally large) characterizing the initial reserves of gas in the reservoir domain $U$.
The constant $A$ is associated with the amount of gas extracted at the well-bore $\Gamma_i$.
Thus the quantity $B-At$ quantifies the gas reserves in the reservoir at the moment $t$.
\end{remark}


For gas filtration in porous media, engineers define the Productivity Index as, see \cite{ aziz-gas-transient, PI-gas}:
\begin{definition}\label{PI-gas-def}
Let $p(x,t)$ be a classical solution of equation \ref{gas-equation} 
satisfying boundary conditions \ref{bc-gas:dirichlet} and \ref{bc-gas:neum}. 
Let $Q(t)$ be the the total mass flow through the well-bore boundary $\Gamma_i$
\begin{equation}\label{gas-total-flux}
Q(t)=\int_{\Gamma_i} \rho u \cdot N \, ds = \int_{\Gamma_i} K_2(p|\n p|) p \nabla p\cdot N\,ds,
\end{equation}
where $N$ is an outward normal to $\Gamma_i$. 
Then productivity index for the well-reservoir system is defined by 
\begin{equation}\label{def:pi-gas}
 J=\cfrac{Q(t)}{\frac1{|U|}\int_U p^2\,dx-\frac1{|\Gamma_i|}\int_{\Gamma_i} p^2\,ds}.
\end{equation}
\end{definition}

As we saw earlier, in case of slightly compressible flow under the boundary conditions \eqref{bc-gas:dirichlet} and \eqref{bc-gas:neum} the value of the PI, defined as in \eqref{def:pi} is stabilizing to a constant value \eqref{def:pi-pss}, which is defined by the PSS solution and depends only on the paramerter $A$. Obvioulsy, in case of the gas flow descirbed by \eqref{gas-equation} this feature is not valid anymore and the PI is not stabilizing in time.
Nevertheless the approach developed in previous sections is applicable to the gas flow as well. Namely we will introduce special auxiliary pressure function characterized by the time independent PI. We  will then study the relation between the general  time dependent PI and this time independent one.  

Let $Q_0=\frac{A}{|U|}$ be a given constant mass-rate of gas production
and consider the auxiliary pressure distribution given by
\begin{equation}\label{p_0def}
 p_0(x,t)=\sqrt{(B-At)^2+2W(x)},
\end{equation}
where $W(x)$ is the solution of the following BVP 
\begin{align}
  &-\nabla \cdot \left( K_2(|\nabla W|)\nabla W \right)= 
  -\nabla\cdot\left(\frac{2\nabla W }{\alpha+\sqrt{\alpha^2+4\beta |\nabla W|}}\right)=A, \label{eq:gas-W}\\
    &W = 0, \quad \text{on}\quad\Gamma_i, \label{eq:gas-Wb} \\
    &\nabla W\cdot N=0, \quad \text{on}\quad\Gamma_e. \label{eq:gas-Wc}
\end{align}
Integrating Eq. \eqref{eq:gas-W} over the domain $U$, 
using integration by parts and boundary condition 
\eqref{eq:gas-Wc} yields to the integral flux condition 
\begin{equation}
\int_{\Gamma_i}  -K_2(|\nabla W|) \nabla W \cdot N \, ds = |U| A = Q_0. \label{fluxW}  
\end{equation}
From \eqref{p_0def} we can also derive that
\begin{equation} p_0 \nabla p_0= \nabla W. \label{gradW} \end{equation}
Combining this property, definition \eqref{p_0def} and BVP  \eqref{eq:gas-W}-\eqref{eq:gas-Wc}
it can be easily verified that the auxiliary pressure $p_0$ satisfies the following equations
\begin{align}
&\frac{\partial p_0}{\partial t}-L[p_0]=f_0(x,t)\quad\text{in $U$,} \label{gas:eq:p0}\\
&p_0(x,t) = B -At\quad\text{on $\Gamma_i$,} \\
&p_0 \nabla p_0 \cdot N =0 \quad\text{on $\Gamma_e$.}\label{gas:eq:p0c}
\end{align}
where 
\begin{align}\label{gas:def-f0}
f_0(x,t)&= \frac{ \partial p_0}{\partial t} - \nabla \cdot \left( K_2( p_0 |\nabla p_0|) p_0\nabla p_0\right) \\
& = \frac{- A( B-At ) }{\sqrt{ (B-At)^2+2W } } -  \nabla \cdot \left( K_2( |\nabla W|) \nabla W \right)\notag \\
& = \frac{- A( B-At ) }{\sqrt{ (B-At)^2+2W } } + A\notag \\
& = A \left(1-\frac{ B-At  }{\sqrt{ (B-At)^2+2W } } \right)>0.\notag
\end{align}
From \eqref{fluxW} and \eqref{gradW}, it also follows the total mass flux condition for $p_0$
\begin{equation}
\int_{\Gamma_i}  -K_2(p_0|\nabla p_0|) p_0\nabla p_0 \cdot N \, ds = \int_{\Gamma_i}  -K_2(|\nabla W|) \nabla W \cdot N \, ds = Q_0. \label{fluxp0}  
\end{equation}

Using this last result and the explicit representation of $p_0(x,t)$ 
in formula \eqref{def:pi-gas} leads to the following Proposition.
\begin{proposition}
The Productivity Index for gas flow defined on  $p_0(x,t)$ is time independent 
and is given by
\begin{equation} \label{jp0}
 J[p_0]=\frac{Q_0}{\frac1{|U|}\int_U2W(x)\,dx}.
\end{equation}
\end{proposition}

Numerical computations, performed for several basic reservoir geometries, 
show that if the initial data in the system \eqref{eq:2forch}-\eqref{bc-gas:neum}
is given by $p(x,0)=p_0(x,0)$, then the corresponding 
productivity indices $J[p_0]$ and $J[p](t)$ are almost identical 
for a long time, see Fig. \ref{fig:PIgas}, as long as the quantity 
\begin{equation}\label{time-constr-1}
 \frac{ B-At  }{\sqrt{ (B-At)^2+2W }}\sim 1,
\end{equation}
or, equivalently, as long as
\begin{equation}\label{time-constr-2}
 (B-At)^2\gg 2\|W\|_{\infty}\quad \Leftrightarrow \quad t \leq T_{crit}< \dfrac{B -  \sqrt{2 \| W\|_\infty }}{A}.
\end{equation}

\begin{figure}[!t]
 \includegraphics [scale = 0.8]{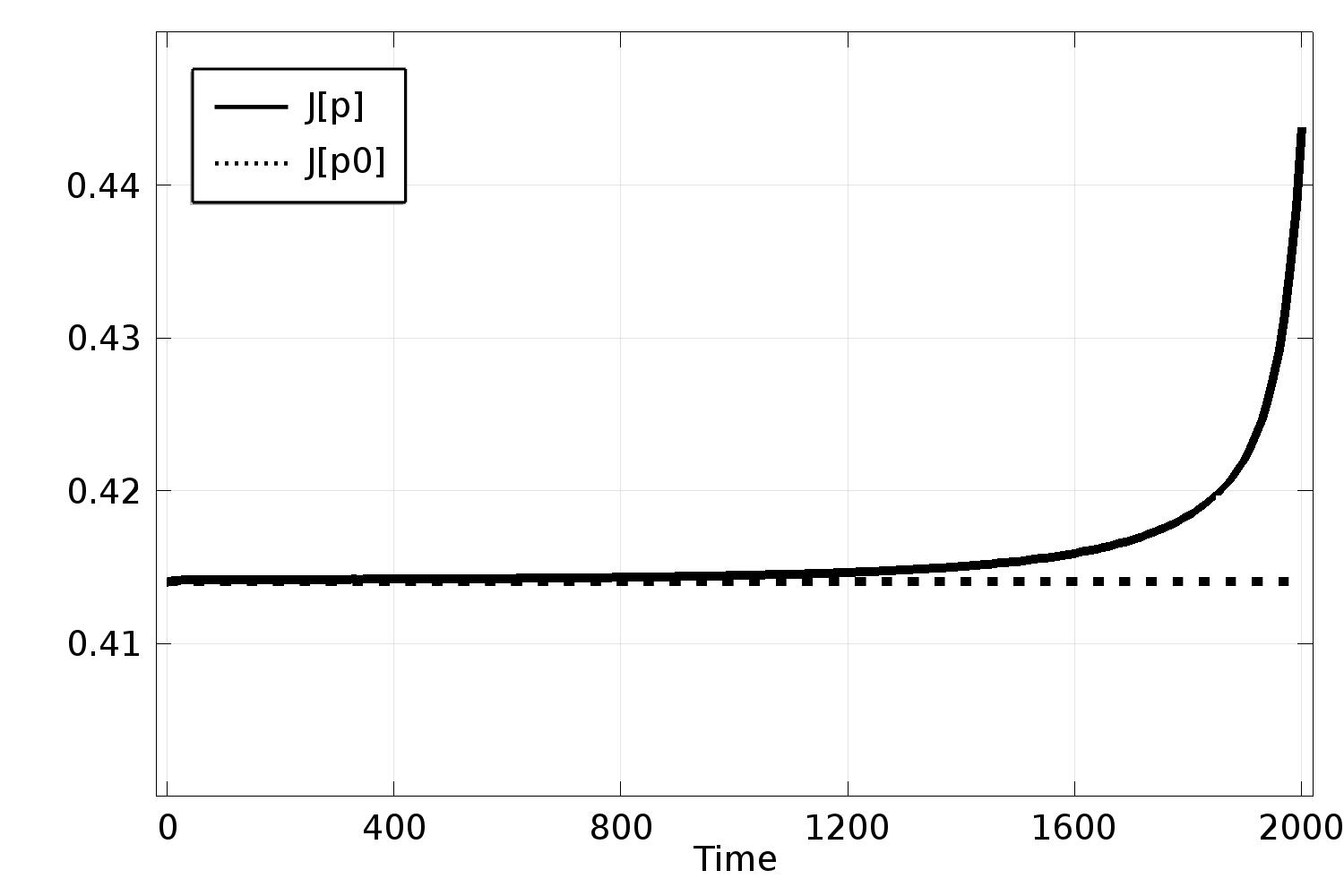}
 \caption{Comparison between the time dependent productivity index $J[p]$ and the PSS productivity index $J[p_0]$
 as the time $t$ approaches the critical value $T_{crit} = \frac{B}{A}=2000$. }
 \label{fig:PIgas}
\end{figure}

In   case of compressible flow, we have to fix a critical time 
$$ T_{crit} = \frac{B}{A}, $$ 
 and for $t>T_{crit}$ the negative boundary data \eqref{bc-gas:dirichlet} leads 
to the violation of ellipticity. 
This behavior is qualitatively justified by the fact that as long as \eqref{time-constr-1} holds, function $f_0(x,t)$ in
the RHS of \eqref{gas:eq:p0} is negligible and $p(x,t)$ and $p_0(x,t)$ behave
similarly. 
On the other hand, as $(B-At)\rightarrow 0$ then $f_0(x,t)$
approaches the constant value $A$, and the two solutions diverge from each other.
Then two productivity indices  $J[p_0]$ and $J[p](t)$ also diverge from each other.

This phenomenon is observed on the actual field data and has a clear practical explanation.
Notice that the denominator in the formula for $J[p_0]$ \eqref{jp0} 
$$ \frac1{|U|} \int_U 2W(x)\,dx = \frac{2 \|W\|_{1}}{| U | }$$
(since $W\ge0$ for any $x\in U$) 
is a measure of the pressure drawdown needed to maintain constant production $Q_0$. 
As long as the gas reserves are considerably larger than the the pressure drawdown  
\begin{equation}
 (B-At)^2 \ge C > 2 \|W\|_\infty >\frac{2 \|W\|_{1}}{| U | },
\end{equation}
then the distributed source term $f_0$ (equivalent to reservoir fluid injection) 
needed to maintain constant production $Q_0$ is negligible. 
Otherwise when the gas reserves are comparable in magnitude
with the pressure drawdown, then a possible way to maintain 
constant production rate is by resupplying the reservoir 
by fluid injections.

In the remaining part of this section we will theoretically investigate the difference between the functions  $p(x,t)$ (the actual solution of the problem) and $p_0(x,t)$  (the solution of the auxiliary problem) depending on the key parameter $T_{crit}=B/A$. 
Unfortunately for the Forchheimer case we were not yet able to obtain the appropriate estimates for the differences between $p_0(x,t)$ and $p(x,t)$. We will report mathematically rigorous result only  for the case of compressible Darcy flow. 
We will show that for the fixed $T_0$ for all times $t\in [0,T_{0}]$ when $p(x,t)>0$ on $\overline{U\times[0,T_{0}]}$ the productivity indices  $J[p]$ and $J[p_0]$   are becoming closer to each other with the increasing parameter $B$. The obtained results  make us believe this comparison can be proved in the general Forchheimer case as well.


We consider the case of positive solutions. For that assume that our time space domain belongs to the time layer $0\le t\le T_0<T_{crit}$. Let $D=U\times(0,T_0]$.   
The following useful inequality follows directly from maximum principle

\begin{lemma}
Let $p(x,t)>0$ for all $0\le t\le T_0<T_{crit}$ is  a classical solution of the equation \eqref{gas-equation} with boundary conditions \eqref{bc-gas:dirichlet} and \eqref{bc-gas:neum}. Then for any $0\le t\le T_0$ 
\begin{equation}\label{p>B-At}p(x,t)\ge \min\{p(x,0); B-At\}.\end{equation}
\end{lemma}

\begin{proof}
Inequality \eqref{p>B-At} follows from the maximum principle (see, for example, \cite{DiBenedetto}), since the nonlinear equation \eqref{gas-equation} is uniformly parabolic in the domain $D$, and the boundary function in \eqref{bc-gas:dirichlet} is decreasing with time.
\end{proof}

%

%

\subsection {Analytical comparison of the solutions for case of the Darcy flow}    
First the following maximum principle follows from the results in \cite{VazquezPorousBook}.
\begin{lemma}\label{lem:comparison}
Let $(x,t)\subset D = U\times(0,T]$, $U\subset \mathbb{R}^n$ and an elliptic operator $\mathcal{L}$ is defined on $U\times(0,T]$
\begin{equation}\label{L-operator}
\mathcal{L}= \sum_{i,j}^n a_{i,j}(x,t)\frac{\partial^2 }{\partial x_i\partial x_j}+\sum_i^n b_i(x,t)\frac{\partial }{\partial x_i}
\end{equation}
with $ C^{-1}|\xi|^2 \sum_{i,j=1}^n a_{i,j} \xi_i\xi_j \geq C |\xi|^2$, and $\sum_{i=1}^n |b_i|\le C$.

Let $\partial U=\Gamma_1\cup \Gamma_2$, where for  simplicity $\Gamma_1$ and $\Gamma_2$ are nonintersecting compact sets.  Let  $\partial D=U\times \{0\} \cup \partial U \times (0,T] $ be the parabolic boundary of the domain. 

Assume $\bar{u} (x,t)$ and $u(x,t)>0$ in $\bar{D}$ are the solution of inequality \eqref{u-bar-eq} and equation \eqref{u-eq} correspondingly:
\begin{align}
\frac{\partial \bar{u}}{\partial t}-  \mathcal{L}[\bar{u}^2]>0,  \label{u-bar-eq} \\
\frac{\partial u}{\partial t}-  \mathcal{L}[u^2]=0.  \label{u-eq} 
\end{align} 
 Assume also that $u(x,t)$ and $\bar{u}(x,t)$  satisfy the homogeneous Neumman conditions on the boundary $\Gamma_2:$
 \begin{equation}\label{neumann-on-gamma2}
  \left.\frac{\partial u}{\partial n}\right|_{\Gamma_2}=\left. \frac{\partial \bar{u}}{\partial n}\right|_{\Gamma_2}=0.
 \end{equation}

\indent  The following comparison principle holds:   if 
 \begin{equation}\label{bc-comparison}
 \bar{u}\ge u  \ \ \text{on} \ \ U\times\{0\}\cup \Gamma_1\times(0,T],
 \end{equation}
then 
\begin{equation}\label{comparison}
 \bar{u}\ge u \ \text{in} \ D.
 \end{equation}
\end{lemma}
%
%
%

\vspace{0.2cm}

We will now obtain some integral comparison results between actual and auxiliary pressures in case of linear Darcy flow.
Namely, let the Forchheimer coefficient $F=0$  in \eqref{eq:2forch-1}.
Let $p_0(x,t)$ be the auxiliary pressure given in Eq.~\eqref{p_0def} and
$p(x,t)$  be the classical  solution of IBVP \eqref{gas-equation} 
with the PSS boundary conditions \eqref{bc-gas:dirichlet}-\eqref{bc-gas:neum}
and initial data given by $p(x,0)=p_0(x,0)=\sqrt{B^2+2W(x)}$. 
First, we will prove a useful integral identity for the difference $p(x,t)-p_0(x,t)$.



\begin{lemma}\label{lemma-identity}
Suppose $T<T_{crit}=B/A$. Then the following identity holds
\begin{align}\label{gas:p-p0-identity}
&\int_0^{T}\int_U (p+p_0)(p-p_0)^2 \,dx\,dt + \frac14\int_U\left(\nabla \int_0^T\big(p^2-p_0^2\big)\,dt \right)^2 dx\\
&\qquad= -\int_0^{T}\int_U \left( f_0(x,t)\int_t^{T}\big(p^2 -p_0^2\big)\,d\tau\right) dx\,dt\,,\notag 
\end{align}
where $f_0$ is defined in \eqref{gas:def-f0}.
The identity above can be further rewritten as
\begin{align}\label{gas:p-p0-identity-1}
&\int_0^{T}\int_U (p+p_0)(p-p_0)^2 \,dx\,dt + \frac14\int_U\left(\nabla \int_0^T\big(p^2-p_0^2\big)\,dt \right)^2 dx \\
&\qquad= \int_U \int_0^{T} \left( \big(p_0^2(x,t) -p^2(x,t)\big) \int_0^{t}f_0(x,\tau)d\tau\right)dt\, dx\,. \notag
\end{align}
\end{lemma}
\begin{proof}
For the Darcy case the equations for $p(x,t)$ and $p_0(x,t)$ take the form
\begin{align}
&\frac{\partial p}{\partial t} = \nabla\cdot(p\nabla p)=\frac12\Delta(p^2)\,,\label{gas:darcy-p} \\
&\frac{\partial p_0}{\partial t} = \nabla\cdot(p_0\nabla p_0)+f_0(x,t)=\frac12\Delta(p_0^2)+f_0(x,t)\,.\label{gas:darcy-p0}
\end{align}
Subtracting the second equation from the first one yields
$$
  \frac{\partial}{\partial t}(p-p_0)=\frac12\Delta\big(p^2-p_0^2\big) - f_0(x,t).
$$
 Following the ideas of Oleinik (see \cite{VazquezPorousBook}), we multiply both sides by the test function $\int_t^T \big(p^2-p_0^2 \big)\,d\tau$, 
 integrate in time form $0$ to $T$ and in space over the domain $U$. Then it follows
 \begin{align}\label{eq:mult-by-test}
  \int_U\int_0^T&\left(\frac{\partial}{\partial t}(p-p_0) \int_t^T\big(p^2-p_0^2\big)\,d\tau \right)\,dt\,dx\\
  =&\frac12\int_U  \int_0^T \left( \Delta\big(p^2-p_0^2\big) \int_t^T\big(p^2-p_0^2\big)\,d\tau\right)\,dt\,dx \notag \\
  &- \int_U\int_0^T\left(f_0 \int_t^T\big(p^2-p_0^2\big)\,d\tau\right) dt\,dx.\notag
 \end{align}
By using integration by part and the fact that $p(x,0)-p_0(x,0)=0$ and $\int_T^T(\cdot) dt=0$,
the integral on the LHS of \eqref{eq:mult-by-test} can be rewritten as
 \begin{align}\label{est:gas-lhs}
  \int_U\int_0^T&\left(\frac{\partial}{\partial t}(p-p_0) \int_t^T\big(p^2-p_0^2\big)\,d\tau \right)\,dt\,dx\\
  &=-\int_U\int_0^T(p-p_0)\cdot\frac{\partial}{\partial t}\left(\int_t^T(p^2-p_0^2)\,d\tau\right)\,dt\,dx\notag\\
    &=\int_U\int_0^T(p-p_0)\big(p^2-p_0^2\big)\,dt\,dx=\int_U\int_0^T(p+p_0)(p-p_0)^2\,dt\,dx\notag.
 \end{align}
By using the divergence theorem the first integral in the RHS of \eqref{eq:mult-by-test} can be rewritten as
  \begin{align}\label{est:gas-rhs}
  &\frac12\int_U  \int_0^T \left( \Delta\big(p^2-p_0^2\big) \int_t^T\big(p^2-p_0^2\big)\,d\tau\right)\,dt\,dx\\
  &=-\frac12\int_U\int_0^T \left( \nabla \big(p^2-p_0^2\big)\cdot \int_t^T\nabla\big(p^2-p_0^2\big)\,d\tau\right)\,dt\,dx\notag\\
  &=\frac14\int_U\int_0^T\frac{\partial}{\partial t} \left( \int_t^T\nabla\big(p^2 -p_0^2\big)\,d\tau\right)^2 \,dt\,dx\notag
   -\frac14\int_U\left[ \int_0^T\nabla\big(p^2-p_0^2\big)\,d\tau\right]^2\,dx\notag.
 \end{align}
Substituting \eqref{est:gas-lhs}  and \eqref{est:gas-rhs} back in \eqref{eq:mult-by-test} we obtain \eqref{gas:p-p0-identity}.
Finally in order to obtain the alternative identity \eqref{gas:p-p0-identity-1} we use the following integration by part
for the right hand side of Eq. \eqref{gas:p-p0-identity}
\begin{align*}
     -\int_U \int_0^{T} &\left( f_0(x,t)\int_t^{T}\big(p^2(x,\tau) -p_0^2(x,\tau)\big)\,d\tau\right) dt\,dx \\
 = & -\int_U \Bigg[ \Bigg(\left. \int_0^t f_0(x,\tau) d\tau \int_t^{T}\big(p^2(x,\tau) -p_0^2(x,\tau) \big)d\tau \right|_0^T \Bigg)\notag\\
   &-\int_0^{T} \left(  \int_0^t f_0(x,\tau) d\tau \right) \left[-\big(p^2(x,t) -p_0^2(x,t) \big)\right] dt\,\Bigg]dx \notag \\
 = &  -\int_U \int_0^{T} \left( \big(p^2(x,t) -p_0^2(x,t)\big) \int_0^{t}f_0(x,\tau)d\tau\right)dt\, dx\,. \notag
\end{align*}

\end{proof}

From the above Lemma follows 
 
 \begin{proposition}
Under the conditions of Lemma \ref{lemma-identity} then  the following comparison holds
\begin{equation}\label{gas:p-p02-f02}
  \left[\int_0^T\!\!\int_U(p^2-p_0^2)\,dx\,d\tau\right]^2\!\leq
 C\int_U\!\left[\int_0^T(p^2-p_0^2)\,d\tau\right]^2\! dx\leq
 C\int_U\! \left(\int_0^{T}f_0(x,t)\,dt\right)^2\!dx.
\end{equation}
\end{proposition}
\begin{proof}

Since $f_0(x,t)>0$ for all $(x,t)$ (see \eqref{gas:def-f0}) and according to Lemma \ref{lem:comparison} $p_0^2(x,t)-p^2(x,t)>0$  for all $(x,t)$, in the RHS of \eqref{gas:p-p0-identity-1} we have 
\begin{align}\label{gas:p-p0-estimate}
\int_U \int_0^{T} &\left( \big(p_0^2 -p^2\big) \int_0^{t}f_0(x,\tau)d\tau\right)dt\, dx
\leq  \int_U \int_0^{T} (p_0^2 -p^2)\int_0^{T}f_0(x,\tau)\,d\tau\,dt\, dx\notag
\\
&\qquad =\int_U \int_0^{T} (p_0^2 -p^2)\,dt \int_0^{T}f_0(x,t)\,dt\, dx\notag\\
&\qquad \leq \e\int_U \left(\int_0^{T} (p_0^2 -p^2)\,dt\right)^2\,dx+C_{\e}\int_U \left(\int_0^{T}f_0(x,t)\,dt\right)^2\, dx.
\end{align}

By Poincar\'{e} inequality we have
\begin{align}\label{gas-poincare}
  \left[\int_0^T\!\!\int_U\!(p^2-p_0^2)\,dx\,d\tau\right]^2\!\!\!\leq
 C\int_U\left[\int_0^T\!(p^2-p_0^2)\,d\tau\right]^2\!\!dx\leq
 C\int_U\left[\nabla \int_0^T\!\!(p^2-p_0^2)\,d\tau\right]^2\!\!dx.
\end{align}

Finally, estimating the RHS of   \eqref{gas:p-p0-identity-1} using \eqref{gas:p-p0-estimate}, neglecting the first term in LHS, and using \eqref{gas-poincare} for the second one, we get:
\begin{align*}
 \left[\int_0^T\!\!\int_U(p^2-p_0^2)\,dx\,d\tau\right]^2\!\!\leq
 \e\int_U \left(\int_0^{T} (p_0^2 -p^2)\,dt\right)^2\!\!dx+C_{\e}\int_U \left(\int_0^{T}f_0(x,t)\,dt\right)^2\!\! dx.
\end{align*}
Choosing appropriate $\e$ we get \eqref{gas:p-p02-f02}.
\end{proof}

\vspace{0.1cm}

\subsection{Stability of PI  with respect to initial and  boudnary  data}

In this section we will show that for the given time interval $0\le t\le T_0<T_{crit}$ 
the difference 
between the PI's for the actual and the auxiliary problems becomes small as the parameter $B$ for the initial reserves becomes large. 
As before we consider only the linear 
Darcy case with $F=0$  in \eqref{eq:2forch-1}.

 Let $p(x,t)$  be the classical  solution of IBVP \eqref{gas-equation}-\eqref{bc-gas:neum} with the corresponding productivity index $J[p](t)$  defined by \eqref{def:pi-gas}. Let $p_0(x,t)$ be the auxiliary pressure given by \eqref{p_0def} with the corresponding productivity index $J[p_0]$  defined by \eqref{jp0}. In linear case the original pressure $p(x,t)$ inherits the following properties of the auxiliary pressure $p_0(x,t)$ (for the details see \cite{VazquezPorousBook}): for all $t$,  $0\le t\le T_0<T_{crit}$ 
\begin{align}
  &|\nabla p(x,t)|+\Delta p(x,t)\le C <\infty;\label{constr-deriv}\\
  &|p_t(x,t)|\leq C<\infty.\label{constr-pt}
\end{align}


In the following two lemmas we will obtain estimates 
for $p-p_0$ and $p_t-p_{0,t}$ as $B \rightarrow 0$.
These will be used later to prove Theorem \ref{theo:gas-pi}. 


\begin{lemma}\label{gas:p-p0to0}
Under the constraints \eqref{constr-deriv} for $B>1$
\begin{equation}
 0\ge p-p_0 \ge -\frac{C}{B^2} \quad \mbox{and }\quad 0\ge p^2-p_0^2 \ge -\frac{C}{B}.
\end{equation}
Moreover, if $B\to\infty$
\begin{equation}
 p-p_0\to 0\quad \mbox{and }\quad p^2-p_0^2\to 0.
\end{equation}

\end{lemma}
\begin{proof}
 Let $a(x,t)=p+p_0$, $b_i(x,t)=\partial(p+p_0)/\partial x_i$, $i=1,\dots,n$ and $c(x,t)=\Delta (p+p_0)$.
The direct calculations show that the function $u=p-p_0$ is a solution of 
\begin{align}
&u_t-a\Delta u+\sum_{i=1}^{n} b_i u_{x_i} +c u=-f_0(x,t)\quad \text{in}\ \ U\times(0,T_0], \label{ueq}\\
&u=0 \quad \text{on} \ \ \partial U\times(0,T_0]. \label{uBC}
\end{align}
Under the conditions \eqref{constr-deriv} it follows that $|b_i(x,t)|+c(x,t)\leq C$ independently of time $t$.
Eq.~\eqref{ueq} is parabolic in $u=p-p_0$ and 
\begin{equation}\label{a-estimate}
 C_1B>a(x,t)>C_0B
\end{equation}
for some  $C_0, C_1 >0$. Thus following the standard arguments using the barrier functions (see \cite{LandisBook}) we get that
\begin{equation} \label{estimate p-p0}
- C\max_{0\le t \le T_0} f_0(x,t) \le p-p_0 \le 0.
\end{equation}
Taking $T_0\le \frac12 T_{crit}=\frac{B}{2A}$, one can get that
$$f_0\le \frac{A}{\sqrt{B^2+2W}+ \frac1{8W}\,B^2}\leq \frac{C}{1+B^2}\to 0\ \text{as}\ B\to\infty.$$
Then the result follows from \eqref{estimate p-p0} and \eqref{a-estimate}.
\end{proof}


\begin{lemma}\label{gas:pt-p0t-to0}
 Under the constraints \eqref{constr-deriv}  and \eqref{constr-pt}
 \begin{equation}
  p_t-p_{0,t}\to 0\quad\text{ as}\quad B\to\infty.
 \end{equation}
\end{lemma}
\begin{proof}
 Multiplying equations \eqref{gas:darcy-p} and \eqref{gas:darcy-p0} on $p$ and $p_0$ correspondingly we get 
\begin{align*}
&2p p_t -p\Delta(p^2)=0\,, \\
&2p_0 p_{0,t} -p_0\Delta(p_0^2)=f_0p_0\,.
\end{align*}
Let $u=p p_t=\frac12(p^2)_t$ and $u_0=p_0 p_{0,t}=\frac12(p_0^2)_t$. Then differentiating the equations above in $t$ we get
\begin{align*}
&2u_t -p\Delta u=p_t\Delta(p^2)\,, \\
&2u_{0,t} -p_0\Delta u_0=p_{0,t}\Delta(p_0^2)+2f_{0,t}p_0+2f_0p_{0,t}\,.
\end{align*}
Since $2p_t =\Delta(p^2)$ and $2 p_{0,t}-f_0=\Delta(p_0^2)$ we can rewrite
\begin{align*}
&2u_t -p\Delta u=2 p^2_t\,, \\
&2u_{0,t} -p_0\Delta u_0=2p^2_{0,t}+2f_{0,t}p_0+f_0p_{0,t}\,.
\end{align*}
Subtracting two last equations from each other we get
\begin{equation}
 2\frac{\partial}{\partial t}\left(u-u_0\right)-p\Delta u+p_0\Delta u_0=2(p^2_t-p^2_{0,t})-F_1(x,t)\label{eq:u-u0}
\end{equation}
where $F_1(x,t)=2f_{0,t}p_0+f_0p_{0,t}$.
Denoting $z=u-u_0$ and adding on both sides of \eqref{eq:u-u0} the term $p\Delta u_0$ we get
\begin{equation*}
 2\frac{\partial z}{\partial t}-p\Delta z=2(p^2_t-p^2_{0,t})-(p_0-p)\Delta u_0-F_1(x,t).
\end{equation*}
Further, since $p_t=u/p$ and $p_{0,t}=u_0/p_0$ we have
\begin{equation}\label{eq:z-1}
 2\frac{\partial z}{\partial t}-p\Delta z=2\left(\frac{u}{p}-\frac{u_0}{p_0}\right)\left(p_t+p_{0,t}\right)-(p_0-p)\Delta u_0-F_1(x,t).
\end{equation}
In the first term in RHS  adding and subtracting the term $u_0p_0$ in the numerator we get
\begin{equation*}
 \frac{u}{p}-\frac{u_0}{p_0}=\frac{p_0u-u_0p}{pp_0}=\frac{p_0(u-u_0)}{pp_0}-\frac{u_0(p-p_0)}{pp_0}
 =\frac{z}{p}-\frac{(p-p_0)}{p}p_{0,t}.
\end{equation*}
Thus from \eqref{eq:z-1} we have linear equation for $z$
\begin{equation*}
 2\frac{\partial z}{\partial t}-p\Delta z=
 2\dfrac{z}{p}\,C(x,t)+F_2(x,t)-F_1(x,t),
\end{equation*}
where
\begin{align*}
&C(x,t)=\left(p_t+p_{0,t}\right),\\
& F_2(x,t)=(p-p_0)\left(\Delta u_0-\frac{2p_{0,t}}{p}(p_t+p_{0,t})\right).
\end{align*}
 Under conditions \eqref{constr-deriv} and \eqref{constr-pt} and in view of Lemma~\ref{gas:p-p0to0}
 $|C(x,t)|\leq C<\infty$ and $F_2(x,t)+F_1(x,t)\to0$ as $B\to\infty$.
 Then $|z|\to 0$ as $B\to\infty$.
The result follows since 
$$z=u-u_0=pp_t-p_0p_{0,t}=p_0(p_t-p_{0,t})-p_t(p_0-p).$$
\end{proof}


Finally, we will state the main theorem of this section.
\begin{theorem}\label{theo:gas-pi}
Under constraints \eqref{constr-deriv} and \eqref{constr-pt} 
\begin{equation}\label{gas:cond-compar-pi}
 J[p](t)-J[p_0]\to 0\quad \mbox{ for } B\to\infty.
\end{equation}
\end{theorem}
\begin{proof}
%
From the boundary condition $p|_{\Gamma_i}=B-At$ and the fact that $(B-At)^2=p_0^2(x,t)-2W(x)$ (see \eqref{p_0def}) we get
 \begin{align}
  J[p_0]&-J[p](t)
  =-\cfrac{Q(t)}{\frac1{|U|}\int_U p^2\,dx-\frac1{|\Gamma_i|}\int_{\Gamma_i} p^2\,ds}
  +\frac{Q_0}{\frac1{|U|}\int_U2W(x)\,dx}\notag\\
  &=-\cfrac{-\int_{\Gamma_i}p\nabla p\cdot N\,ds}{\frac1{|U|}\int_U p^2\,dx-(B-At)^2}
  +\frac{-\int_{\Gamma_i}p_0\nabla p_0\cdot N\,ds}{\frac1{|U|}\int_U2W(x)\,dx}\notag\\  
   &=\cfrac{\frac{d}{dt}\int_U p(x,t)\,dx}{\frac1{|U|}\int_U (p^2-p_0^2)\,dx+\frac1{|U|}\int_U2W(x)\,dx}
   -\frac{\frac{d}{dt}\int_{U}p_0(x,t)\,dx-\int_Uf_0(x,t)\,dx}{\frac1{|U|}\int_U2W(x)\,dx}\notag
 \end{align}
 Adding and subtracting the term $\frac{d}{dt}\int_{U}p_0(x,t)\,dx$ in the numerator of the first fraction we get
  \begin{align}
  &(J[p_0]-J[p](t))\cdot |U|= \cfrac{\frac{d}{dt}\int_U (p-p_0)\,dx}{\int_U (p^2-p_0^2)\,dx+2\int_UW\,dx} +\frac{\int_Uf_0(x,t)\,dx}{2\int_UW\,dx}\notag\\
   &+\frac{d}{dt}\left(\int_U p_0\,dx\right)
   \left[\cfrac{1}{\int_U (p^2-p_0^2)\,dx+2\int_UW\,dx}   -\frac{1}{2\int_UW\,dx}\right]
  \notag\\ 
  &=  \cfrac{\frac{d}{dt}\int_U (p-p_0)\,dx}{\int_U (p^2-p_0^2)\,dx+2\int_UW\,dx}
  +\frac{\frac{d}{dt}\left(\int_U p_0\,dx\right)}{2\int_UW\,dx}\cdot
   \cfrac{-\int_U (p^2-p_0^2)\,dx}{\int_U (p^2-p_0^2)\,dx+2\int_UW\,dx}\notag\\
   &\hspace{8.cm}+\frac{\int_Uf_0(x,t)\,dx}{2\int_UW\,dx}.\notag
 \end{align}
 Finally
   \begin{align}\label{jp-jp0}
  (J[p_0]-J[p](t))\cdot |U|&= \frac{\int_Uf_0(x,t)\,dx}{2\int_UW\,dx}+ \frac{1}{\int_U (p^2-p_0^2)\,dx+2\int_UW\,dx}\times\notag\\
  &
  \times\,\left[\frac{d}{dt}\int_U (p-p_0)\,dx
  -\int_U (p^2-p_0^2)\,dx\,\cdot\,\frac{\frac{d}{dt}\left(\int_U p_0\,dx\right)}{2\int_UW\,dx} \right].
 \end{align}
 
\noindent Since since $f_0\to 0$ as $B\to\infty$ (see \eqref{gas:def-f0}) then
 \begin{align*}
   \frac{\int_Uf_0(x,t)\,dx}{2\int_UW(x)\,dx}\to 0 \quad \text{as $B\to\infty$}.
 \end{align*}
Then \eqref{gas:cond-compar-pi} follows from Lemma~\ref{gas:p-p0to0} and Lemma~\ref{gas:pt-p0t-to0}.
\end{proof}



\vspace{0.1cm}

\section{Conclusions}
\begin{itemize}
 \item The notion of diffusive capacity/productivity index is studied in case of Forchheimer flow of slightly compressible and strongly compressible fluids. In general case the PI is a time dependent integral functional over the pressure function.
 
  \item In case of slightly compressible fluid we consider two types of boundary profile on the well-boundary: the total flux condition and Dirichlet boundary condition. In both cases we prove the convergence of the time dependent PI to a constant value without any constraints on the degree of nonlinearity of Forchheimer polynomial. This generalizes our previous work \cite{ABI12}, requiring the estimates on the mixed term $|\n p|^{2-a}|p|^{\ah-2}$, $\ah\neq 2$, and resulting in stronger constraints on smoothness of boundary data. 
  
  \item In case of strongly compressible fluid, the ideal gas, we  study the Dirichlet boundary problem. The quantity $B-At$  prescribed on the well-boundary specifies the gas reserves at the moment $t$.   We show numerically that the PI stays the same until it suddenly blows up when time approaches the critical value $T_{crit}=B/A$. This fact corresponds with field observations. We associate the constant PI with the special pressure $p_0(x,t)$ and in case of linear flow of the ideal gas we analytically study the relation between general pressure and $p_0(x,t)$. The results on stability of the PI with respect to the initial gas reserves are obtained. 
\end{itemize}

\section{Acknowledgments}
The research of the first and third authors was partially supported by 
the National Science Foundation grant DMS-1412796.


\bibliography{reference}{}
\bibliographystyle{plain}

%
%
%
%

\end{document}